\documentclass[11pt]{article}

\usepackage[T1]{fontenc}
\usepackage[utf8]{inputenc}
\usepackage[english]{babel}
\usepackage{lmodern}
\usepackage{amsmath,amssymb,amsthm,mathtools}
\usepackage{microtype}
\usepackage{geometry}
\usepackage{booktabs,tabularx,array}
\usepackage{float}
\usepackage{tikz}
\usepackage[numbers,sort&compress]{natbib}
\usepackage[colorlinks=true,linkcolor=blue,citecolor=blue,urlcolor=blue]{hyperref}
\hypersetup{
  pdftitle={Selections of Set-Valued Maps under Stieltjes Clocks: Regularity, Variation, and Atomic Structure},
  pdfauthor={Serkan Ilter; Hulya Duru; Seyit Koca},
  pdfsubject={Selections of set-valued maps under Stieltjes clocks},
  pdfkeywords={set-valued map, selection, Stieltjes clock, Holder regularity, bounded variation, Hausdorff metric, Riesz p-variation, jump structure}
}

\newtheorem{theorem}{Theorem}[section]
\newtheorem{proposition}[theorem]{Proposition}
\newtheorem{lemma}[theorem]{Lemma}
\newtheorem{corollary}[theorem]{Corollary}
\theoremstyle{definition}
\newtheorem{definition}[theorem]{Definition}
\newtheorem{example}[theorem]{Example}
\theoremstyle{remark}
\newtheorem{remark}[theorem]{Remark}

\newcommand{\K}{\mathcal K}
\newcommand{\VH}{V_H}
\newcommand{\Var}{V}
\newcommand{\Jump}{\Delta^{+}}

\title{Selections of Set-Valued Maps under Stieltjes Clocks: Regularity, Variation, and Atomic Structure}
\author{%
\large Serkan \.{I}lter$^{1,*}$ \qquad
\large H\"ulya Duru$^{2}$ \qquad
\large Seyit Koca$^{3}$\\[7pt]
\normalsize $^{1,2}$Department of Mathematics, Faculty of Science, Istanbul University, Istanbul, T\"urkiye\\
\normalsize $^{3}$Department of Management Information Systems, Istinye University, Istanbul, T\"urkiye\\[5pt]
\small \texttt{ilters@istanbul.edu.tr} \qquad
\texttt{hduru@istanbul.edu.tr} \qquad
\texttt{seyit.koca@istinye.edu.tr}\\[5pt]
\footnotesize $^{*}$Corresponding author: Serkan Ilter, \texttt{ilters@istanbul.edu.tr}%
}
\date{}

\begin{document}
\maketitle

\begin{abstract}
A Stieltjes clock allows effective time to advance continuously, remain unchanged over intervals, or jump. We study whether a compact-valued set-valued map evolving relative to a Stieltjes clock admits a single-valued selection that passes through a prescribed graph point while retaining the regularity and variation of the multifunction.

For $g$-H\"older exponents $\alpha\ge1$, we prove that regularity can be preserved without increasing the $g$-H\"older seminorm, with no monotonicity assumption on $g$. If $g$ is additionally nondecreasing, one prescribed-point selection preserves both this regularity and the Hausdorff variation on every subinterval. The same selection consequently preserves all finite Riesz $p$-variations associated with nondecreasing external clocks.

For $\alpha>1$, the structure becomes jump-driven. For left-continuous nondecreasing Stieltjes clocks, continuous clock evolution cannot generate variation: all variation is carried by jumps. We obtain exact jump decompositions for both the set-valued map and its selection, together with explicit atomic formulas for Riesz $p$-variation. Examples show that the principal regularity, variation, and jump bounds are attained.

For compact-convex Euclidean-valued maps, we also examine the case $0<\alpha<1$. In this setting, the H\"older exponent can still be preserved, but preservation of the same constant through a prescribed point holds in one dimension and can fail in higher dimensions. Finally, without a quantitative H\"older bound, we characterize exactly when every compact-valued Hausdorff $g$-continuous map admits a prescribed-point $g$-continuous selection: precisely when the clock image is zero-dimensional.
\end{abstract}



\section{Introduction}\label{sec:introduction}

Set-valued models allow several states to be admissible at the same time. A natural question is whether one can choose a single admissible state at each time while preserving the regularity and variation of the set-valued map. We study this problem when change is measured through a Stieltjes clock that can advance continuously, remain constant on intervals, or jump. We also require the selection to pass through a prescribed point of the graph. Our main question is therefore when these properties can be preserved simultaneously by one selection.

The role of the clock becomes especially visible for $g$-H\"older regularity with exponent $\alpha>1$. A continuous clock forces such a map to be constant, whereas jumps can support nonconstant behavior. Thus the case $\alpha>1$ becomes nontrivial precisely because a Stieltjes clock may contain jumps. Example~\ref{ex:one-jump} gives a simple instance of this phenomenon.

Classical selection results already provide strong preservation statements at the Lipschitz and bounded-variation levels. For compact-valued maps on an interval, prescribed-point selections can be chosen without increasing the Lipschitz constant, and bounded-variation selections can be obtained with corresponding variation estimates \cite{Chistyakov1998,ChistyakovGalkin1998}. Later work extended bounded-variation and generalized-variation selection results to more general settings \cite{BelovChistyakov2000,Chistyakov2004,ChistyakovRepovs2007}. These results make Hausdorff variation a natural quantity for comparing a multifunction with its selections.

Below the Lipschitz threshold the picture changes. Chistyakov and Galkin showed that, for every $0<\gamma<1$, a compact-valued $\gamma$-H\"older multifunction on an interval may have no continuous selection at all \cite{ChistyakovGalkin1998}. For convex-valued maps, selection results are available in finite-dimensional Euclidean spaces: standard convex selections, including Steiner-type selections, preserve H\"older exponents through their Lipschitz dependence on the Hausdorff metric \cite{Dentcheva2000,DeutschLiPark1989,ArutyunovObukhovskii2017}. Related modulus-of-continuity problems for nonconvex multifunctions also arise in differential inclusions \cite{BressanAncona1993}.

Under the same compact-convex assumptions used in classical H\"older selection results, we sharpen the problem by requiring the selection to pass through a prescribed graph point. Let $0<\alpha<1$ and let $F:I\to\mathcal K_c(\mathbb R^n)$ be $g$-H\"older. The H\"older exponent can still be preserved. In one dimension, the prescribed point can be retained without increasing the original $g$-H\"older constant. From dimension two onward, this same-constant property may fail, while a prescribed-point selection with the same exponent remains available with an explicit controlled enlargement of the constant.

At and above the Lipschitz threshold, the assumptions needed to preserve $g$-H\"older regularity and local variation are different. For $\alpha\ge1$, $g$-H\"older regularity can be preserved by a prescribed-point selection without assuming that $g$ is monotone, and the $g$-H\"older seminorm does not increase. Local variation requires more: for a general nonmonotone $g$, the same selection need not satisfy the corresponding variation bound on every subinterval. When $g$ is nondecreasing, one prescribed-point selection preserves both the $g$-H\"older regularity and the Hausdorff-variation control locally.

The same local variation control also leads to preservation of Riesz $p$-variation. We consider Riesz $p$-variation with respect to a nondecreasing external Stieltjes clock, extending the usual weighted formulation to clocks that may contain plateaus and jumps. A zero increment of the external clock prevents motion when the Riesz $p$-variation is finite, whereas a jump of the clock can support a corresponding state jump with finite contribution. The prescribed-point selection obtained above preserves all finite Riesz $p$-variations of this type.

Classical bounded-variation results already relate jumps to the variation function. Here the additional $g$-H\"older condition with exponent $\alpha>1$ gives a stronger conclusion. For a left-continuous nondecreasing Stieltjes clock, the continuous part of the clock contributes no Hausdorff variation, and the variation of the set-valued map is exactly the sum of its jumps. The prescribed-point selection obtained above has the corresponding jump decomposition, with each of its jumps bounded by the jump of the set-valued map at the same point. This also yields explicit atomic formulas for the associated Riesz $p$-variation.

The examples make the change at the Lipschitz threshold explicit. At $\alpha=1$, continuous variation may persist, whereas for $\alpha>1$ clock jumps can support nonconstant behavior. A Stieltjes clock may also have a continuous component even though only its jumps contribute to the variation of a $g$-H\"older map with exponent greater than one. Finally, a two-point example attains the regularity, variation, and jump bounds simultaneously.

Finally, we consider $g$-continuity without a quantitative H\"older bound. Since a $g$-continuous set-valued map is constant on the level sets of $g$, it factors through $g(I)$. The selection problem can therefore be studied as an ordinary continuous selection problem on the subspace $g(I)\subset\mathbb R$. We show that, for complete metric targets, every compact-valued Hausdorff $g$-continuous map admits a prescribed-point $g$-continuous selection exactly when $g(I)$ is zero-dimensional. If $g(I)$ contains a nondegenerate interval, the classical nonconvex selection failure reappears. Pure-jump Stieltjes clocks form an important special case, but the zero-dimensionality criterion also applies beyond clocks whose Stieltjes measure is purely atomic.

The paper is organized as follows. After the preliminaries, Sections~\ref{sec:finite-chain}--\ref{sec:main-preservation} establish prescribed-point selection results and the preservation of $g$-H\"older regularity and local variation. Section~\ref{sec:riesz} treats Riesz $p$-variation with respect to external Stieltjes clocks. Sections~\ref{sec:atomic-power}--\ref{sec:exact-jumps} describe the jump structure for $g$-H\"older maps with exponent $\alpha>1$, including the atomic power-sum identity and the exact decomposition of variation into jumps. Section~\ref{sec:examples} gives examples that attain the main bounds. Sections~\ref{sec:convex-sub} and~\ref{sec:g-continuous} consider, respectively, convex $g$-H\"older selections below the Lipschitz threshold and $g$-continuous selections characterized by the zero-dimensionality of $g(I)$. The final section collects the main conclusions.

\section{Preliminaries and notation}\label{sec:preliminaries}

Let $I=[a,b]$ be a compact interval and let $(X,d)$ be a metric space. Throughout the paper, we write
\[
 \Omega_I:=\{(s,t)\in I^2:s\le t\}.
\]
Denote by $\K(X)$ the family of all nonempty compact subsets of $X$. For $x\in X$ and $B\subset X$, set $d(x,B)=\inf_{y\in B}d(x,y)$. For $A,B\in\K(X)$, the Hausdorff distance is
\begin{equation}\label{eq:hausdorff}
 d_H(A,B)
 =\max\left\{
 \sup_{x\in A}d(x,B),
 \sup_{y\in B}d(y,A)
 \right\}.
\end{equation}
No convexity or linear structure on $X$ will be assumed in the main quantitative results.

A set-valued map $F:I\to\K(X)$ is called constant on a set $J\subset I$ if there exists $K\in\K(X)$ such that $F(t)=K$ for every $t\in J$.

\begin{definition}[$g$-regularity]\label{def:g-regularity}
Let $g:I\to\mathbb R$ and let $\alpha>0$. A map $F:I\to\K(X)$ is called $g$-H\"older of exponent $\alpha$ if there exists $L\ge0$ such that, for all $s,t\in I$,
\begin{equation}\label{eq:g-holder}
 d_H(F(s),F(t))\le L|g(t)-g(s)|^\alpha.
\end{equation}
For a $g$-H\"older multifunction $F:I\to\K(X)$, $[F]_{\alpha,g}$ denotes the minimum constant $L\ge0$ for which \eqref{eq:g-holder} holds. For $\alpha=1$, we say that $F$ is $g$-Lipschitz.
\end{definition}

Because \eqref{eq:g-holder} is imposed for all pairs, $g(s)=g(t)$ implies $F(s)=F(t)$. Thus $F$ is constant on each set $g^{-1}(\{c\})$, $c\in g(I)$. This fact will also be used in Sections~\ref{sec:main-preservation} and~\ref{sec:g-continuous}. No monotonicity of $g$ is required in Definition~\ref{def:g-regularity}; monotonicity will be imposed only when variation or Stieltjes-measure structure is needed.

\begin{definition}[$g$-continuity]\label{def:g-continuity}
A map $F:I\to\K(X)$ is called $g$-continuous if, for every $t\in I$ and every $\varepsilon>0$, there exists $\delta>0$ such that $d_H(F(s),F(t))<\varepsilon$ whenever $s\in I$ and $|g(s)-g(t)|<\delta$. For a single-valued map $f:I\to X$, $g$-continuity is defined analogously, with $d$ in place of $d_H$.
\end{definition}

For a map $g:I\to\mathbb R$, we write $\operatorname{Im}(g)=g(I)$ for its image on $I$. When regarded as a map $g:I\to\operatorname{Im}(g)$, $g$ is surjective.

A \emph{selection} of $F$ is a map $f:I\to X$ satisfying $f(t)\in F(t)$ for every $t\in I$. For a single-valued $g$-H\"older map $f:I\to X$, $[f]_{\alpha,g}$ denotes the minimum constant $L\ge0$ such that $d(f(s),f(t))\le L|g(t)-g(s)|^\alpha$ for all $s,t\in I$. If $\theta\in I$ and $x_\theta\in F(\theta)$ are fixed, we call $f$ a prescribed-point selection if it additionally satisfies $f(\theta)=x_\theta$.

For $a\le s<t\le b$, the Hausdorff variation of $F$ on $[s,t]$ is
\begin{equation}\label{eq:hausdorff-variation}
 \VH(F;[s,t])
 =\sup_P\sum_{i=1}^n d_H(F(t_{i-1}),F(t_i)),
\end{equation}
where $P$ ranges over all finite partitions $s=t_0<\cdots<t_n=t$. Similarly,
\begin{equation}\label{eq:single-variation}
 \Var(f;[s,t])
 =\sup_P\sum_{i=1}^n d(f(t_{i-1}),f(t_i)).
\end{equation}
When $\VH(F;[a,b])<\infty$, the variation function
\[
 \Phi_F(t)=\VH(F;[a,t])
\]
is nondecreasing and
\begin{equation}\label{eq:variation-additivity}
 \VH(F;[s,t])=\VH(F;[s,u])+\VH(F;[u,t])
 \qquad (s\le u\le t).
\end{equation}
Also,
\begin{equation}\label{eq:endpoint-by-variation}
 d_H(F(s),F(t))\le \VH(F;[s,t]).
\end{equation}

When $g$ is left-continuous and nondecreasing, its right jump at $\tau<b$ is
\begin{equation}\label{eq:right-jump-g}
 \Jump g(\tau)=g(\tau+)-g(\tau).
\end{equation}
For such a clock, we use the associated Stieltjes measure $\mu_g$ with the convention
\[
 \mu_g([s,t))=g(t)-g(s).
\]
In particular,
\[
 \mu_g(\{\tau\})=\Jump g(\tau).
\]
A point $\tau<b$ is a jump point of $g$ if $\Jump g(\tau)>0$. Equivalently, $\tau$ is an atom of the Stieltjes measure $\mu_g$, with mass $\mu_g(\{\tau\})=\Jump g(\tau)$. We call $g$ \emph{pure-jump} when its Stieltjes measure $\mu_g$ is purely atomic, meaning that $\mu_g$ is concentrated on its atoms. Equivalently, for every $t\in I$,
\[
 g(t)-g(a)=\sum_{\tau\in[a,t)}\Jump g(\tau).
\]
In Sections~\ref{sec:atomic-power}--\ref{sec:exact-jumps}, where the jump structure is studied, we additionally assume that $g$ is left-continuous and nondecreasing and that $\alpha>1$.

\section{Controlled selections on finite chains}\label{sec:finite-chain}

The finite-chain result below provides the estimate used in the compactness argument of the next section.

\begin{lemma}[Prescribed-point selection on a finite chain]\label{lem:finite-chain}
Let $F:I\to\K(X)$, and let $Q:\Omega_I\to[0,\infty)$ satisfy, for every $a\le s\le t\le b$,
\begin{equation}\label{eq:Q-dominates}
 d_H(F(s),F(t))\le Q(s,t),
\end{equation}
and, for every triple $(s,u,t)\in I^3$ satisfying $a\le s\le u\le t\le b$,
\begin{equation}\label{eq:Q-superadditive}
 Q(s,u)+Q(u,t)\le Q(s,t).
\end{equation}
Fix $\theta\in I$ and $x_\theta\in F(\theta)$. Let
\[
 t_0<t_1<\cdots<t_m,\qquad m\ge1,
\]
be points of $I$, and suppose that $t_k=\theta$ for some $k\in\{0,\ldots,m\}$. Then there exist points $x_i\in F(t_i)$, $i=0,\ldots,m$, with $x_k=x_\theta$, such that, for any indices $0\le i<j\le m$,
\begin{equation}\label{eq:finite-chain-bound}
 d(x_i,x_j)\le Q(t_i,t_j).
\end{equation}
\end{lemma}

\begin{proof}
Set $x_k=x_\theta$. For $i=k,\ldots,m-1$, compactness of $F(t_{i+1})$ allows us to choose $x_{i+1}\in F(t_{i+1})$ minimizing the distance from $x_i$. Hence
\[
 d(x_i,x_{i+1})
 =d(x_i,F(t_{i+1}))
 \le d_H(F(t_i),F(t_{i+1}))
 \le Q(t_i,t_{i+1}).
\]
Proceed backward in the same way from $x_k$. For any indices $0\le i<j\le m$, the triangle inequality and repeated use of \eqref{eq:Q-superadditive} give
\[
 d(x_i,x_j)
 \le\sum_{r=i+1}^j Q(t_{r-1},t_r)
 \le Q(t_i,t_j).
\]
\end{proof}

The following two controls will be used repeatedly in later sections. For $\alpha\ge1$, $L\ge0$, and a nondecreasing map $g:I\to\mathbb R$, define
\[
 Q_{\alpha,g}:\Omega_I\to[0,\infty),\qquad
 Q_{\alpha,g}(s,t)=L[g(t)-g(s)]^\alpha.
\]
If $F$ has bounded Hausdorff variation, define
\[
 Q_V:\Omega_I\to[0,\infty),\qquad
 Q_V(s,t)=\VH(F;[s,t]).
\]
Both controls satisfy \eqref{eq:Q-superadditive}. For $Q_{\alpha,g}$ this follows from
\[
 u^\alpha+v^\alpha\le(u+v)^\alpha,\qquad u,v\ge0,
\]
while $Q_V$ satisfies \eqref{eq:Q-superadditive} with equality by \eqref{eq:variation-additivity}.

\begin{remark}
For a finite chain $t_0<t_1<t_2<t_3$, with $t_2=\theta$ and increments $a_i=g(t_i)-g(t_{i-1})$, nearest-point choices give
\[
 d(x_0,x_3)
 \le L(a_1^\alpha+a_2^\alpha+a_3^\alpha)
 \le L(a_1+a_2+a_3)^\alpha.
\]
This is the finite-chain estimate used in the global selection argument.
\end{remark}

\section{Global selection via Tychonoff compactness}\label{sec:tychonoff}

The finite-chain lemma provides compatible choices on every finite ordered set of times. Tychonoff compactness and the finite intersection property then yield a single selection on the whole interval satisfying all the corresponding constraints.

\begin{theorem}[Prescribed-point selection under a superadditive control]\label{thm:controlled-selection}
Let $F:I\to\K(X)$, let $\theta\in I$, and let $x_\theta\in F(\theta)$. Let
\[
 Q:\Omega_I\to[0,\infty)
\]
satisfy \eqref{eq:Q-dominates} and \eqref{eq:Q-superadditive}. Then there exists a selection $f:I\to X$ such that
\[
 f(\theta)=x_\theta
\]
and, for every $a\le s<t\le b$,
\begin{equation}\label{eq:global-Q}
 d(f(s),f(t))\le Q(s,t).
\end{equation}
\end{theorem}

\begin{proof}
Let $\mathcal P=\prod_{t\in I}F(t)$, endowed with the product topology. Since every $F(t)$ is compact, $\mathcal P$ is compact by Tychonoff's theorem. For $a\le s<t\le b$, let $C_{s,t}$ consist of those $f\in\mathcal P$ satisfying $d(f(s),f(t))\le Q(s,t)$, and let $C_\theta$ consist of those $f\in\mathcal P$ satisfying $f(\theta)=x_\theta$. These sets are closed because the coordinate projections are continuous.

Take a finite subfamily of these constraint sets. It involves only finitely many time points. After including $\theta$ if necessary, write them as $t_0<t_1<\cdots<t_m$, with $\theta=t_k$ for some $k$. Lemma~\ref{lem:finite-chain}, applied with $x_k=x_\theta$, provides points $x_i\in F(t_i)$, $i=0,\ldots,m$, such that $d(x_i,x_j)\le Q(t_i,t_j)$ whenever $0\le i<j\le m$. Hence all constraints in the chosen finite subfamily, together with the prescribed-point condition, are satisfied.

For each remaining $t\in I\setminus\{t_0,\ldots,t_m\}$, choose any $x_t\in F(t)$. These choices define an element of $\mathcal P$ belonging to every set in the chosen finite subfamily. Thus the family $\{C_\theta\}\cup\{C_{s,t}:a\le s<t\le b\}$ has the finite intersection property.

Since $\mathcal P$ is compact and all the constraint sets are closed,
\[
 C_\theta\cap\bigcap_{a\le s<t\le b}C_{s,t}\ne\varnothing.
\]
Any $f$ in this intersection is a selection of $F$, satisfies $f(\theta)=x_\theta$, and obeys $d(f(s),f(t))\le Q(s,t)$ for every $a\le s<t\le b$.
\end{proof}

\begin{theorem}[Prescribed-point selection with local Hausdorff-variation control]\label{thm:local-BV-prescribed}
Let $F:I\to\K(X)$ have bounded Hausdorff variation. Then, for every $\theta\in I$ and every $x_\theta\in F(\theta)$, there exists a selection $f:I\to X$ with $f(\theta)=x_\theta$ such that, for every $a\le s<t\le b$,
\begin{equation}\label{eq:local-BV-control}
\begin{aligned}
 d(f(s),f(t)) &\le \VH(F;[s,t]),\\
 \Var(f;[s,t]) &\le \VH(F;[s,t]).
\end{aligned}
\end{equation}
\end{theorem}

\begin{proof}
Use the control $Q_V$ defined in Section~\ref{sec:finite-chain}. By \eqref{eq:endpoint-by-variation}, it satisfies \eqref{eq:Q-dominates}, and by \eqref{eq:variation-additivity}, it satisfies \eqref{eq:Q-superadditive} with equality. Theorem~\ref{thm:controlled-selection} therefore gives a prescribed-point selection satisfying the stated distance bound. For any partition $s=t_0<\cdots<t_n=t$,
\[
 \sum_{i=1}^n d(f(t_{i-1}),f(t_i))
 \le\sum_{i=1}^n\VH(F;[t_{i-1},t_i])
 =\VH(F;[s,t]).
\]
Taking the supremum over all partitions proves the stated variation bound.
\end{proof}

\begin{remark}[Relation with the classical BV selection theorem]\label{rem:chistyakov-BV}
Prescribed-point selections of bounded variation with a global estimate
\[
 \Var(f;[a,b])\le \VH(F;[a,b])
\]
are well established in the literature; see, in particular, Chistyakov's Theorem~6.1(b) \cite{Chistyakov1998}. Chistyakov's construction proceeds through finite approximants and a Helly-type compactness argument, whereas the result above is obtained by a different argument. It is the $Q=Q_V$ specialization of Theorem~\ref{thm:controlled-selection}, whose finite-chain constraints are extended to the whole interval by the finite intersection property and Tychonoff compactness.

Chistyakov's Theorem~6.1(b), as stated, gives the variation estimate on the full interval and does not state that the same selection satisfies the corresponding estimate simultaneously on every subinterval. Theorem~\ref{thm:local-BV-prescribed} provides this simultaneous local control: one global prescribed-point selection $f$ satisfies the estimates in \eqref{eq:local-BV-control} for every $[s,t]\subset[a,b]$. Applying the classical theorem separately to individual subintervals would not by itself provide this property for one and the same selection. Thus the Hausdorff variation of $F$ locally controls both the displacement and the total variation of a single prescribed-point selection. Later, this local domination will imply preservation, by the same selection, of every finite Stieltjes--Riesz energy associated with an admissible external clock.
\end{remark}

\section{Preservation of \texorpdfstring{$g$}{g}-H\"older regularity and local variation}\label{sec:main-preservation}

At and above the Lipschitz threshold, $g$-H\"older regularity and local variation require different assumptions for their preservation. For $\alpha\ge1$, a prescribed-point selection can preserve $g$-H\"older regularity without assuming monotonicity of $g$. Local variation requires an additional assumption on $g$: for a general nonmonotone $g$, the corresponding local variation bound may fail, whereas for nondecreasing $g$ one prescribed-point selection preserves both properties.

\begin{theorem}[Regularity preservation without monotonicity]\label{thm:arbitrary-clock-regularity}
Let $(X,d)$ be a metric space, let $g:I\to\mathbb R$, and let $F:I\to\K(X)$. Suppose that, for some $\alpha\ge1$ and $L\ge0$, the inequality
\begin{equation}\label{eq:arbitrary-clock-holder}
 d_H(F(s),F(t))\le L|g(t)-g(s)|^\alpha
\end{equation}
holds for all $s,t\in I$. Then, for every $\theta\in I$ and every $x_\theta\in F(\theta)$, there exists a selection $f:I\to X$ such that $f(\theta)=x_\theta$ and
\begin{equation}\label{eq:arbitrary-clock-selection}
 d(f(s),f(t))\le L|g(t)-g(s)|^\alpha
\end{equation}
for all $s,t\in I$. Consequently, if $L=[F]_{\alpha,g}$, then
\begin{equation}\label{eq:arbitrary-clock-seminorm}
 [f]_{\alpha,g}\le [F]_{\alpha,g}.
\end{equation}
\end{theorem}

\begin{proof}
Since \eqref{eq:arbitrary-clock-holder} holds for all $s,t\in I$, the equality $g(s)=g(t)$ implies $F(s)=F(t)$. Hence the map $\widehat F:g(I)\to\K(X)$ defined by $\widehat F(g(t)):=F(t)$ for $t\in I$ is well defined. Moreover, for all $u,v\in g(I)$,
\begin{equation}\label{eq:image-holder}
 d_H(\widehat F(u),\widehat F(v))
 \le L|u-v|^\alpha.
\end{equation}
Put $u_\theta=g(\theta)\in g(I)$. Consider distinct points $u_0,\ldots,u_m\in g(I)$ satisfying $u_0<u_1<\cdots<u_m$, with $u_\theta=u_k$ for some $k\in\{0,\ldots,m\}$. Set $x_k=x_\theta\in\widehat F(u_k)$. Proceeding to the right and to the left as in Lemma~\ref{lem:finite-chain}, compactness of the values $\widehat F(u_i)$ permits nearest-point choices $x_i\in\widehat F(u_i)$. For every $r=1,\ldots,m$, these choices satisfy
\[
 d(x_{r-1},x_r)
 \le d_H(\widehat F(u_{r-1}),\widehat F(u_r))
 \le L(u_r-u_{r-1})^\alpha.
\]
Therefore, for $0\le i<j\le m$,
\begin{align*}
 d(x_i,x_j)
 &\le L\sum_{r=i+1}^j (u_r-u_{r-1})^\alpha\\
 &\le L\left(\sum_{r=i+1}^j(u_r-u_{r-1})\right)^\alpha\\
 &=L(u_j-u_i)^\alpha,
\end{align*}
where the second inequality uses $\alpha\ge1$. Thus every finite ordered subset of $g(I)$ containing $u_\theta$ admits a prescribed-point selection satisfying all the required pairwise bounds.

Now consider $\widehat{\mathcal P}=\prod_{u\in g(I)}\widehat F(u)$, endowed with the product topology. We apply the Tychonoff finite-intersection argument used in the proof of Theorem~\ref{thm:controlled-selection}, now on $g(I)$. For $u,v\in g(I)$ with $u<v$, impose
\[
 C_{u,v}=\{h\in\widehat{\mathcal P}:d(h(u),h(v))\le L|u-v|^\alpha\},
\]
together with $C_\theta=\{h\in\widehat{\mathcal P}:h(u_\theta)=x_\theta\}$. These sets are closed. Any finite family of such constraints involves only finitely many points of $g(I)$; after adjoining $u_\theta$ if necessary, the finite-chain construction above shows that their intersection is nonempty. Hence, by the same compactness argument as in Theorem~\ref{thm:controlled-selection}, there exists a selection $\widehat f:g(I)\to X$ such that $\widehat f(u)\in\widehat F(u)$ for every $u\in g(I)$, $\widehat f(u_\theta)=x_\theta$, and
\[
 d(\widehat f(u),\widehat f(v))\le L|u-v|^\alpha
\]
for all $u,v\in g(I)$.

Finally, define $f:I\to X$ by $f(t):=\widehat f(g(t))$ for $t\in I$. Then $f(t)\in F(t)$ for every $t\in I$, $f(\theta)=x_\theta$, and \eqref{eq:arbitrary-clock-selection} holds for all $s,t\in I$.
\end{proof}

\begin{remark}\label{rem:arbitrary-clock-variation}
Theorem~\ref{thm:arbitrary-clock-regularity} concerns preservation of $g$-H\"older regularity by a single prescribed-point selection. This estimate alone does not imply the local variation bound $\Var(f;[s,t])\le\VH(F;[s,t])$. In particular, for an arbitrary clock $g$, $g$-H\"older regularity need not even imply finite classical variation. If $F$ has bounded Hausdorff variation, Theorem~\ref{thm:local-BV-prescribed} still provides a prescribed-point selection satisfying this local variation bound without requiring monotonicity of $g$; however, that selection need not be the one preserving the $g$-H\"older seminorm. Example~\ref{ex:nonmonotone-variation-obstruction} shows that, when $g$ is not monotone, the two properties cannot in general be imposed simultaneously on the same selection. When $g$ is nondecreasing, Theorem~\ref{thm:main-preservation} provides one prescribed-point selection preserving both properties.
\end{remark}

\begin{theorem}[Simultaneous regularity and variation preservation]\label{thm:main-preservation}
Let $(X,d)$ be a metric space, let $g:I\to\mathbb R$ be nondecreasing, and let $F:I\to\K(X)$. Suppose that, for some $\alpha\ge1$,
\[
 [F]_{\alpha,g}<\infty.
\]
Then, for every $\theta\in I$ and every $x_\theta\in F(\theta)$, there exists a selection $f:I\to X$ with $f(\theta)=x_\theta$ such that, for every $s,t\in I$ with $s<t$,
\begin{equation}\label{eq:main-preservation}
 d(f(s),f(t))
 \le \VH(F;[s,t])
 \le [F]_{\alpha,g}|g(t)-g(s)|^\alpha,
\end{equation}
and
\begin{equation}\label{eq:variation-preservation}
 \Var(f;[s,t])\le \VH(F;[s,t]).
\end{equation}
Consequently,
\begin{equation}\label{eq:seminorm-preservation}
 [f]_{\alpha,g}\le[F]_{\alpha,g}.
\end{equation}
\end{theorem}

\begin{proof}
Put $L=[F]_{\alpha,g}$. Fix $s,t\in I$ with $s<t$, and let $P=\{t_0,\ldots,t_n\}$ be a finite partition of $[s,t]$ with $s=t_0<\cdots<t_n=t$. Then
\[
 \sum_{i=1}^n d_H(F(t_{i-1}),F(t_i))
 \le L\sum_{i=1}^n[g(t_i)-g(t_{i-1})]^\alpha.
\]
Since $g$ is nondecreasing and $\alpha\ge1$,
\[
 \sum_{i=1}^n[g(t_i)-g(t_{i-1})]^\alpha
 \le[g(t)-g(s)]^\alpha.
\]
Recall from Section~\ref{sec:finite-chain} the controls $Q_{\alpha,g}$ and $Q_V$. Taking the supremum over all finite partitions of $[s,t]$ yields
\begin{equation}\label{eq:VH-from-g}
 Q_V(s,t)
 =\VH(F;[s,t])
 \le L[g(t)-g(s)]^\alpha
 =Q_{\alpha,g}(s,t).
\end{equation}

Theorem~\ref{thm:local-BV-prescribed} gives a prescribed-point selection $f:I\to X$ such that, for every $s,t\in I$ with $s<t$,
\[
 d(f(s),f(t))\le Q_V(s,t)=\VH(F;[s,t])
\]
and
\[
 \Var(f;[s,t])\le\VH(F;[s,t]).
\]
Combining the first estimate with \eqref{eq:VH-from-g} gives
\[
 d(f(s),f(t))
 \le \VH(F;[s,t])
 \le L[g(t)-g(s)]^\alpha,
\]
which proves the asserted $g$-H\"older estimate and hence $[f]_{\alpha,g}\le[F]_{\alpha,g}$.
The second estimate is exactly the stated local variation bound.
\end{proof}

\begin{remark}\label{rem:local}
Estimate \eqref{eq:variation-preservation} holds on every subinterval, not only on $[a,b]$. The same selection satisfies both the regularity and variation bounds.
\end{remark}

\begin{corollary}[Local variation preservation at the $g$-Lipschitz boundary]\label{cor:lipschitz}
Let $(X,d)$ be a metric space, let $g:I\to\mathbb R$ be nondecreasing, and let $F:I\to\K(X)$ satisfy $[F]_{1,g}<\infty$. Then, for every $\theta\in I$ and every $x_\theta\in F(\theta)$, there exists a selection $f:I\to X$ with $f(\theta)=x_\theta$ such that, for every $s,t\in I$ with $s<t$,
\[
 \Var(f;[s,t])\le\VH(F;[s,t])
\]
and
\[
 d(f(s),f(t))\le\VH(F;[s,t])\le [F]_{1,g}[g(t)-g(s)].
\]
In particular, $f$ is $g$-Lipschitz and $[f]_{1,g}\le[F]_{1,g}$.
\end{corollary}

At $\alpha=1$, Theorem~\ref{thm:arbitrary-clock-regularity} already preserves $g$-Lipschitz regularity without assuming monotonicity of $g$, while Corollary~\ref{cor:lipschitz} additionally preserves local variation when $g$ is nondecreasing. The behavior changes when the exponent exceeds one: at $\alpha=1$, nontrivial continuous variation may still occur.

\begin{example}[Continuous variation at $\alpha=1$]\label{ex:alpha-one}
Let $I=[0,1]$, $g(t)=t$, and let $L>0$ and $R>0$. Set
\[
 F(t)=Lt+\{0,R\}.
\]
Then $[F]_{1,g}=L$. The selection $f(t)=Lt$ satisfies
\begin{equation}\label{eq:alpha-one-sharp}
 \Var(f;[s,t])=\VH(F;[s,t])=L(t-s).
\end{equation}
Thus nontrivial continuous variation can persist at the Lipschitz boundary.
\end{example}

\section{Preservation of Riesz \texorpdfstring{$p$}{p}-variation under external clocks}\label{sec:riesz}

Riesz $p$-variation and its weighted variants are well established \cite{MichtaMotyl2022,AzizGuerreroMerentes2013}. Here we consider the corresponding partition functional with physical-time increments replaced by increments of a nondecreasing external clock $h$. This allows plateaus and jumps of the clock to be reflected directly in the variation. If $h$ does not increase on a partition cell, finiteness forces the state to remain unchanged on that cell; a jump of $h$, on the other hand, can support a corresponding state jump with a finite contribution. The local variation estimate obtained in the previous section will be used to show that one prescribed-point selection preserves all finite Riesz $p$-variations of this form.

\begin{definition}[Riesz $p$-variation with respect to an external clock]\label{def:riesz-energy}
Let $p>1$, let $h:I\to\mathbb R$ be nondecreasing, and let $u:I\to Y$ take values in a metric space $(Y,\rho)$. For every $a\le s<t\le b$, define
\[
 \mathcal R_{p,h}(u;[s,t])
 =\sup_{P}\sum_{i=1}^{n}
 \frac{\rho(u(t_i),u(t_{i-1}))^p}
 {[h(t_i)-h(t_{i-1})]^{p-1}},
\]
where the supremum is over all partitions $P:s=t_0<\cdots<t_n=t$. If $h(t_i)=h(t_{i-1})$, the corresponding term is defined to be $0$ when $u(t_i)=u(t_{i-1})$ and $+\infty$ otherwise. For $F:I\to\K(X)$, the Hausdorff Riesz $p$-variation with respect to $h$, denoted by $\mathcal R^H_{p,h}(F;[s,t])$, is defined by taking $Y=\K(X)$ and $\rho=d_H$.

For degenerate intervals, we adopt the convention $\mathcal R_{p,h}(u;[s,s])=0$ and, correspondingly, $\mathcal R^H_{p,h}(F;[s,s])=0$.
\end{definition}

Finite Riesz $p$-variation with respect to $h$ automatically forces constancy on intervals where $h$ is constant: if $h(s)=h(t)$ and $\mathcal R^H_{p,h}(F;[s,t])<\infty$, then $F$ is constant on $[s,t]$.

\begin{proposition}[Additivity and variation control]\label{prop:riesz-additivity}
Let $(X,d)$ be a metric space, let $F:I\to\K(X)$, let $p>1$, and let $h:I\to\mathbb R$ be nondecreasing. Then, for every $a\le s\le u\le t\le b$ such that $\mathcal R^H_{p,h}(F;[s,t])<\infty$, we have
\[
 \mathcal R^H_{p,h}(F;[s,t])
 =\mathcal R^H_{p,h}(F;[s,u])
 +\mathcal R^H_{p,h}(F;[u,t]),
\]
and
\[
 \VH(F;[s,t])
 \le
 \bigl[\mathcal R^H_{p,h}(F;[s,t])\bigr]^{1/p}
 [h(t)-h(s)]^{1-1/p}.
\]
The same conclusions hold for a map $u:I\to Y$ into a metric space $(Y,\rho)$, with $d_H$ replaced by $\rho$.
\end{proposition}

\begin{proof}
The cases $u=s$ or $u=t$, including the degenerate case $s=t$, follow immediately from the convention for degenerate intervals. We may therefore assume $s<u<t$.

Consider first a partition having a cell $[r,v]$ with $r<u<v$. Put
\[
 x=d_H(F(r),F(u)),\qquad y=d_H(F(u),F(v)),
\]
\[
 a=h(u)-h(r),\qquad b=h(v)-h(u).
\]
The triangle inequality together with
\[
 \frac{(x+y)^p}{(a+b)^{p-1}}
 \le
 \frac{x^p}{a^{p-1}}+\frac{y^p}{b^{p-1}}
\]
shows that inserting $u$ does not decrease the Riesz sum; the zero-increment cases follow from the convention in Definition~\ref{def:riesz-energy}. Hence every partition of $[s,t]$ may be refined to contain $u$ without decreasing its sum. Dividing the interval at $u$ gives
\[
 \mathcal R^H_{p,h}(F;[s,t])
 \le
 \mathcal R^H_{p,h}(F;[s,u])
 +\mathcal R^H_{p,h}(F;[u,t]).
\]
The reverse inequality follows by combining arbitrary partitions of $[s,u]$ and $[u,t]$. Hence the stated additivity follows.

For a partition $s=t_0<\cdots<t_n=t$, write $d_i=d_H(F(t_{i-1}),F(t_i))$ and $a_i=h(t_i)-h(t_{i-1})$. Finiteness of $\mathcal R^H_{p,h}(F;[s,t])$ implies $d_i=0$ whenever $a_i=0$. H\"older's inequality on the remaining cells gives
\[
 \sum_i d_i
 \le
 \left(\sum_i\frac{d_i^p}{a_i^{p-1}}\right)^{1/p}
 \left(\sum_i a_i\right)^{1-1/p}.
\]
The first factor is bounded by $[\mathcal R^H_{p,h}(F;[s,t])]^{1/p}$, while $\sum_i a_i=h(t)-h(s)$. Taking the supremum over all partitions yields the stated variation bound. The single-valued case is identical.
\end{proof}

\begin{theorem}[Simultaneous preservation of Riesz $p$-variation under external clocks]\label{thm:universal-riesz}
Let $(X,d)$ be a metric space, let $F:I\to\K(X)$ have bounded Hausdorff variation, and let $f:I\to X$ be a selection of $F$ satisfying $d(f(s),f(t))\le\VH(F;[s,t])$ for every $a\le s<t\le b$. Then, for every $p>1$ and every nondecreasing external clock $h:I\to\mathbb R$ such that $\mathcal R^H_{p,h}(F;I)<\infty$, the same selection satisfies
\[
 \mathcal R_{p,h}(f;[s,t])
 \le
 \mathcal R^H_{p,h}(F;[s,t])
\]
for every $a\le s<t\le b$. Hence the prescribed-point selection supplied by Theorem~\ref{thm:local-BV-prescribed} preserves all finite Riesz $p$-variations of this form.
\end{theorem}

\begin{proof}
Fix $p>1$, a nondecreasing external clock $h$, and $a\le s<t\le b$. By Proposition~\ref{prop:riesz-additivity},
\[
 \VH(F;[s,t])
 \le
 \bigl[\mathcal R^H_{p,h}(F;[s,t])\bigr]^{1/p}
 [h(t)-h(s)]^{1-1/p}.
\]
Applying this estimate on each partition cell with positive $h$-increment and using $d(f(t_{i-1}),f(t_i))\le \VH(F;[t_{i-1},t_i])$ gives
\[
 \frac{d(f(t_{i-1}),f(t_i))^p}
 {[h(t_i)-h(t_{i-1})]^{p-1}}
 \le
 \mathcal R^H_{p,h}(F;[t_{i-1},t_i]).
\]
If the clock increment on a partition cell is zero, finiteness of $\mathcal R^H_{p,h}(F;I)$ forces $F$ to be constant on that cell. The local displacement bound then also forces $f$ to be constant there, so the corresponding contribution is zero. Summing over the partition and using interval additivity from Proposition~\ref{prop:riesz-additivity}, then taking the supremum over all partitions of $[s,t]$, proves the claim.
\end{proof}

The next example shows the role of jumps of the external clock. A state jump has infinite Riesz $p$-variation with respect to physical time, but it may have finite Riesz $p$-variation when the external clock jumps at the same point.

\begin{example}[A jump of the external clock gives finite Riesz $p$-variation]\label{ex:riesz-one-jump}
Let $p>1$, let $m>0$, and set
\[
 h(t)=\begin{cases}
 0,&t\le\tau,\\
 m,&t>\tau,
 \end{cases}
 \qquad
 u(t)=\begin{cases}
 x_0,&t\le\tau,\\
 x_1,&t>\tau,
 \end{cases}
\]
with $J=\rho(x_0,x_1)>0$. With physical time as the reference clock, take a partition cell $[r,v]$ with $r<\tau<v$. Its contribution across the jump is $J^p/(v-r)^{p-1}$. As $r\uparrow\tau$ and $v\downarrow\tau$, this quantity tends to $+\infty$. Hence $u$ has infinite Riesz $p$-variation with respect to physical time.

With respect to the external clock $h$, partition cells on which the state does not change have zero contribution, while the cell containing the jump has a state change of size $J$ and a clock increment of size $m$. Therefore
\[
 \mathcal R_{p,h}(u;I)=\frac{J^p}{m^{p-1}}<\infty.
\]
Thus the same state jump that gives infinite Riesz $p$-variation with respect to physical time has finite Riesz $p$-variation when the external clock has a jump of size $m$ at the same point.
\end{example}

\section{Partition power sums and clock jumps}\label{sec:atomic-power}

For $\alpha>1$, the continuous and jump parts of a Stieltjes clock contribute differently to power sums of its increments. The theorem below shows that, after taking the infimum over finite partitions, the continuous contribution disappears and only the jumps remain.

Throughout this section, let $g:I\to\mathbb R$ be left-continuous and nondecreasing, and let $\alpha>1$.

We use the left-continuous Stieltjes convention
\begin{equation}\label{eq:stieltjes-convention}
 \mu_g([s,t))=g(t)-g(s).
\end{equation}
Hence
\begin{equation}\label{eq:stieltjes-atom}
 \mu_g(\{\tau\})=g(\tau+)-g(\tau)=\Jump g(\tau).
\end{equation}

\begin{lemma}[Small-mass partition]\label{lem:small-mass}
Let $\mu$ be a finite Borel measure on $[u,v)$ and suppose
\[
 \mu(\{x\})\le\delta\qquad(x\in[u,v))
\]
for some $\delta>0$. Then there is a finite partition
\[
 u=r_0<r_1<\cdots<r_N=v
\]
such that
\begin{equation}\label{eq:small-mass-cells}
 \mu([r_{j-1},r_j))\le2\delta,
 \qquad j=1,\ldots,N.
\end{equation}
\end{lemma}

\begin{proof}
See Appendix~\ref{app:small-mass-proof}, \emph{Proof of the small-mass partition lemma}.
\end{proof}

For every $a\le s<t\le b$ and every finite partition
\[
 P:s=t_0<t_1<\cdots<t_n=t
\]
of $[s,t]$, define
\begin{equation}\label{eq:partition-power-sum}
 S_{\alpha,g}(P;[s,t])
 =\sum_{i=1}^n[g(t_i)-g(t_{i-1})]^\alpha.
\end{equation}

\begin{theorem}[Atomic power-sum identity]\label{thm:atomic-power}
Let $g:I\to\mathbb R$ be left-continuous and nondecreasing, and let $\alpha>1$. Then, for every $a\le s<t\le b$,
\begin{equation}\label{eq:atomic-power-identity}
 \inf_P S_{\alpha,g}(P;[s,t])
 =\sum_{\substack{\tau\in[s,t)\\ \Jump g(\tau)>0}}[\Jump g(\tau)]^\alpha.
\end{equation}
Here the infimum is taken over all finite partitions $P$ of $[s,t]$. Equivalently, the sum on the right is taken over the atoms of the Stieltjes measure $\mu_g$ contained in $[s,t)$, or, equivalently, over the jump points of $g$ in that interval. The identity holds without any restriction on the number or distribution of these jumps.
\end{theorem}

\begin{proof}
Fix $a\le s<t\le b$, and write
\[
 m_\tau=\Jump g(\tau),\qquad M=g(t)-g(s).
\]
If $M=0$, then $\mu_g([s,t))=0$, so there are no positive atoms in $[s,t)$ and, for every finite partition $P:s=t_0<\cdots<t_n=t$, one has $g(t_i)-g(t_{i-1})=0$ for $i=1,\ldots,n$. Hence both sides of \eqref{eq:atomic-power-identity} are zero. We may therefore assume $M>0$.

Since
\[
 \sum_{\tau\in[s,t)}m_\tau\le M,
\]
the set of positive atoms in $[s,t)$ is at most countable. Moreover,
\[
 \sum_{\tau\in[s,t)}m_\tau^\alpha
 \le M^{\alpha-1}\sum_{\tau\in[s,t)}m_\tau
 \le M^\alpha,
\]
so the atomic power sum is finite.

For the lower bound, let
\[
 P:s=t_0<\cdots<t_n=t
\]
be an arbitrary finite partition of $[s,t]$, and put $I_i=[t_{i-1},t_i)$. Then
\[
 g(t_i)-g(t_{i-1})=\mu_g(I_i)\ge\sum_{\tau\in I_i}m_\tau.
\]
Since $\alpha>1$ and the terms are nonnegative,
\[
 [g(t_i)-g(t_{i-1})]^\alpha
 \ge\sum_{\tau\in I_i}m_\tau^\alpha.
\]
Summing over all partition cells gives
\[
 S_{\alpha,g}(P;[s,t])\ge\sum_{\tau\in[s,t)}m_\tau^\alpha.
\]
Since $P$ was arbitrary,
\[
 \inf_P S_{\alpha,g}(P;[s,t])
 \ge\sum_{\tau\in[s,t)}m_\tau^\alpha.
\]

For the reverse inequality, fix $\varepsilon>0$. Choose $\delta>0$ sufficiently small that
\begin{equation}\label{eq:delta-choice}
 (2\delta)^{\alpha-1}M<\frac{\varepsilon}{2}.
\end{equation}
The set
\[
 A_\delta=\{\tau\in[s,t):m_\tau>\delta\}
\]
is finite. List its elements as $\tau_1,\ldots,\tau_N$.

For each $\tau_j$, continuity from above of the finite measure $\mu_g$ gives
\[
 \mu_g([\tau_j,r))\longrightarrow\mu_g(\{\tau_j\})=m_{\tau_j}
 \qquad\text{as }r\downarrow\tau_j.
\]
Since $A_\delta$ is finite, the endpoints $r_j$ can therefore be chosen so that the intervals $[\tau_j,r_j)$ are pairwise disjoint, contain no other point of $A_\delta$, and satisfy
\[
 \mu_g([\tau_j,r_j))<m_{\tau_j}+\eta_j,
\]
where the positive numbers $\eta_j$ are chosen sufficiently small that
\begin{equation}\label{eq:large-atoms-approx}
 \sum_{j=1}^N(m_{\tau_j}+\eta_j)^\alpha
 <\sum_{j=1}^Nm_{\tau_j}^\alpha+\frac{\varepsilon}{2}.
\end{equation}
Insert all $\tau_j$ and $r_j$ into a partition of $[s,t]$. The complement of the resulting large-atom cells is a finite union of half-open intervals on which every atom has mass at most $\delta$. Apply Lemma~\ref{lem:small-mass} to each of these intervals.

Thus every residual partition cell $C_i$ satisfies $\mu_g(C_i)\le2\delta$. Consequently,
\[
 \sum_i\mu_g(C_i)^\alpha
 \le(2\delta)^{\alpha-1}\sum_i\mu_g(C_i)
 \le(2\delta)^{\alpha-1}M
 <\frac{\varepsilon}{2}.
\]
The power-sum contribution of each large-atom cell $[\tau_j,r_j)$ is bounded by
\[
 [\mu_g([\tau_j,r_j))]^\alpha
 <(m_{\tau_j}+\eta_j)^\alpha.
\]
Combining the large-atom cells with the residual cells therefore gives a finite partition $P_\varepsilon$ of $[s,t]$ such that
\[
 S_{\alpha,g}(P_\varepsilon;[s,t])
 <\sum_{\tau\in[s,t)}m_\tau^\alpha+\varepsilon.
\]
Hence
\[
 \inf_P S_{\alpha,g}(P;[s,t])
 \le\sum_{\tau\in[s,t)}m_\tau^\alpha.
\]
Together with the lower bound, this proves \eqref{eq:atomic-power-identity}.
\end{proof}

\begin{remark}\label{rem:atomic-interpretation}
Identity \eqref{eq:atomic-power-identity} shows exactly what remains under partition refinement. Nonatomic increments can be subdivided into arbitrarily small increments, and their $\alpha$-power contribution disappears under refinement when $\alpha>1$. A clock jump cannot be removed by interval subdivision and leaves the contribution $[\Jump g(\tau)]^\alpha$. For instance, if a clock increment of size $M$ is spread continuously over $n$ equal cells, its contribution is $M^\alpha n^{1-\alpha}\to0$, whereas a single jump of $g$ of size $M$ contributes the fixed amount $M^\alpha$.
\end{remark}

\section{Variation controlled by clock jumps for \texorpdfstring{$g$}{g}-H\"older maps with exponent \texorpdfstring{$\alpha>1$}{alpha > 1}}\label{sec:atomic-variation}

The atomic power-sum identity converts the $g$-H\"older estimate with exponent $\alpha>1$ into direct bounds for variation. In particular, the continuous part of a nondecreasing Stieltjes clock contributes no variation; only its jumps remain in the resulting estimates.

Throughout this section, let $g:I\to\mathbb R$ be left-continuous and nondecreasing, let $\alpha>1$, and let $F:I\to\K(X)$ be $g$-H\"older of exponent $\alpha$, with $L=[F]_{\alpha,g}<\infty$.

\begin{proposition}[Endpoint estimate in terms of clock jumps]\label{prop:atomic-endpoint}
Let $(X,d)$ be a metric space, let $g:I\to\mathbb R$ be left-continuous and nondecreasing, let $\alpha>1$, and let $F:I\to\K(X)$ be $g$-H\"older of exponent $\alpha$ with $L=[F]_{\alpha,g}<\infty$. Then, for every $a\le s<t\le b$,
\begin{equation}\label{eq:atomic-endpoint}
 d_H(F(s),F(t))
 \le L\sum_{\tau\in[s,t)}[\Jump g(\tau)]^\alpha.
\end{equation}
\end{proposition}

\begin{proof}
Let $P:s=t_0<\cdots<t_n=t$ be an arbitrary finite partition of $[s,t]$. By the triangle inequality and the $g$-H\"older estimate,
\[
 d_H(F(s),F(t))
 \le\sum_{i=1}^n d_H(F(t_{i-1}),F(t_i))
 \le L S_{\alpha,g}(P;[s,t]).
\]
Taking the infimum over all finite partitions $P$ of $[s,t]$ and applying Theorem~\ref{thm:atomic-power} gives \eqref{eq:atomic-endpoint}.
\end{proof}

\begin{theorem}[Variation bound in terms of clock jumps]\label{thm:atomic-variation}
Let $(X,d)$ be a metric space, let $g:I\to\mathbb R$ be left-continuous and nondecreasing, let $\alpha>1$, and let $F:I\to\K(X)$ be $g$-H\"older of exponent $\alpha$. Set $L=[F]_{\alpha,g}<\infty$. Then, for every $a\le s<t\le b$,
\begin{equation}\label{eq:atomic-variation}
 \VH(F;[s,t])
 \le L\sum_{\tau\in[s,t)}[\Jump g(\tau)]^\alpha.
\end{equation}
\end{theorem}

\begin{proof}
Let $P:s=t_0<\cdots<t_n=t$ be an arbitrary finite partition of $[s,t]$. Applying Proposition~\ref{prop:atomic-endpoint} to each partition cell gives
\[
 d_H(F(t_{i-1}),F(t_i))
 \le L\sum_{\tau\in[t_{i-1},t_i)}[\Jump g(\tau)]^\alpha.
\]
The half-open cells $[t_{i-1},t_i)$ are pairwise disjoint and their union is $[s,t)$. Summing over the partition cells and then taking the supremum over all finite partitions gives \eqref{eq:atomic-variation}.
\end{proof}

Together with Theorem~\ref{thm:main-preservation}, this yields for the same prescribed-point selection
\begin{equation}\label{eq:atomic-selection-bound}
 \Var(f;[s,t])
 \le\VH(F;[s,t])
 \le L\sum_{\tau\in[s,t)}[\Jump g(\tau)]^\alpha.
\end{equation}

Define the atomic measure
\begin{equation}\label{eq:atomic-measure}
 \lambda_{\alpha,g}
 =\sum_{\tau\in[a,b)}[\Jump g(\tau)]^\alpha\delta_\tau.
\end{equation}
Its mass at a jump point $\tau$ of $g$ is $\lambda_{\alpha,g}(\{\tau\})=[\Jump g(\tau)]^\alpha$. Let
\begin{equation}\label{eq:atomic-clock}
 A_{\alpha,g}(t)=\lambda_{\alpha,g}([a,t)),\qquad a\le t\le b,
\end{equation}
be the cumulative function associated with $\lambda_{\alpha,g}$. Then, for every $a\le s<t\le b$,
\[
 A_{\alpha,g}(t)-A_{\alpha,g}(s)
 =\lambda_{\alpha,g}([s,t))
 =\sum_{\tau\in[s,t)}[\Jump g(\tau)]^\alpha.
\]

\begin{proposition}[Reduction to the cumulative function of $\lambda_{\alpha,g}$]\label{prop:atomic-clock}
Let $(X,d)$ be a metric space, let $g:I\to\mathbb R$ be left-continuous and nondecreasing, let $\alpha>1$, and let $F:I\to\K(X)$ be $g$-H\"older of exponent $\alpha$ with $L=[F]_{\alpha,g}<\infty$. Fix $\theta\in I$ and $x_\theta\in F(\theta)$, and let $f:I\to X$ be the prescribed-point selection supplied by Theorem~\ref{thm:main-preservation}, with $f(\theta)=x_\theta$. Then both $F$ and $f$ are Lipschitz with respect to $A_{\alpha,g}$, with constant at most $L$. More precisely, for every $a\le s<t\le b$,
\begin{equation}\label{eq:atomic-clock-lip}
 d_H(F(s),F(t))
 \le L|A_{\alpha,g}(t)-A_{\alpha,g}(s)|,
\end{equation}
and
\[
 d(f(s),f(t))
 \le L|A_{\alpha,g}(t)-A_{\alpha,g}(s)|.
\]
\end{proposition}

\begin{proof}
The first estimate follows from Proposition~\ref{prop:atomic-endpoint} and the definition of $A_{\alpha,g}$; the second follows from \eqref{eq:atomic-selection-bound}.
\end{proof}

\begin{corollary}[Constancy on intervals where the clock is continuous]\label{cor:continuous-rigidity}
Let $(X,d)$ be a metric space, let $g:I\to\mathbb R$ be left-continuous and nondecreasing, let $\alpha>1$, and let $F:I\to\K(X)$ be $g$-H\"older of exponent $\alpha$. For every $a\le s<t\le b$, if $g$ is continuous on $[s,t]$, then $F$ is constant on $[s,t]$. In particular, if $g(t)=t$ on $I$, every compact-valued $\alpha$-H\"older map with $\alpha>1$ is constant.
\end{corollary}

\begin{proof}
If $g$ is continuous on $[s,t]$, then $\Jump g(\tau)=0$ for every $\tau\in[s,t)$. Hence Theorem~\ref{thm:atomic-variation} gives $\VH(F;[s,t])=0$, and therefore $F$ is constant on $[s,t]$.
\end{proof}

Table~\ref{tab:regimes} summarizes how the continuous and jump parts of a left-continuous nondecreasing Stieltjes clock contribute to variation at and above the Lipschitz threshold.

\begin{table}[ht]
\centering
\caption{Contribution of the continuous and jump parts of the clock to variation.}
\label{tab:regimes}
\begin{tabularx}{0.92\textwidth}{>{\raggedright\arraybackslash}p{0.18\textwidth} X X}
\toprule
Regime & Continuous part of the clock & Jump part of the clock \\
\midrule
$\alpha=1$ & May contribute to variation & May contribute to variation \\
$\alpha>1$ & Contributes no variation & Carries all variation \\
\bottomrule
\end{tabularx}
\end{table}

\begin{example}[Mixed clock]\label{ex:mixed-clock}
Let $\tau\in(0,1)$ and let $m>0$, $L>0$, and $R>0$. Define
\[
 g(t)=t+m\mathbf 1_{\{t>\tau\}},
\]
where $\mathbf 1_{\{t>\tau\}}$ denotes the indicator of the set $\{t>\tau\}$.
Thus $g$ has a nontrivial continuous component and a single right jump of size $m$ at $\tau$. Let $K=\{0,R\}\subset\mathbb R$ and define
\[
 F(t)=
 \begin{cases}
 K,&t\le\tau,\\
 K+Lm^\alpha,&t>\tau.
 \end{cases}
\]
For $s\le\tau<t$,
\[
 d_H(F(s),F(t))=Lm^\alpha
 \le L[g(t)-g(s)]^\alpha,
\]
while $F$ is constant on either side of $\tau$. Thus $F$ is $g$-H\"older of exponent $\alpha>1$ with constant $L$. Moreover,
\[
 \VH(F;[0,1])=Lm^\alpha.
\]
Thus $g$ need not be pure-jump: although the clock has a continuous component, only its jump contributes to the variation of $F$.
\end{example}

\section{Exact jump structure of the multifunction and its selection}\label{sec:exact-jumps}

The previous section shows that, for $\alpha>1$, the variation of a $g$-H\"older map is controlled by the jumps of the Stieltjes clock. We now identify the exact jump decomposition of the set-valued map and of the prescribed-point selection, and compare their jumps pointwise.

Let $g:I\to\mathbb R$ be left-continuous and nondecreasing, let $\alpha>1$, and let $F:I\to\K(X)$ be $g$-H\"older of exponent $\alpha$, with $L=[F]_{\alpha,g}<\infty$. Assume that $(X,d)$ is complete. Fix $\theta\in I$ and $x_\theta\in F(\theta)$, and let $f:I\to X$ be the prescribed-point selection supplied by Theorem~\ref{thm:main-preservation}, with $f(\theta)=x_\theta$.

Completeness is used only in this section to obtain right limits in $X$ and in the metric space $(\K(X),d_H)$.

Since $g$ is left-continuous, $t_n\uparrow t$ implies
\[
 d_H(F(t_n),F(t))\le L|g(t)-g(t_n)|^\alpha\to0.
\]
The selection $f$ is left-continuous by the same estimate.

\begin{proposition}[Existence of right limits]\label{prop:right-limits}
For every $a\le\tau<b$, the limits
\[
 F(\tau+)=\lim_{r\downarrow\tau}F(r)
 \quad\text{in }(\K(X),d_H),
\]
and
\[
 f(\tau+)=\lim_{r\downarrow\tau}f(r)
 \quad\text{in }X
\]
exist.
\end{proposition}

\begin{proof}
Fix $a\le\tau<b$. If $\tau<r_1<r_2\le b$, then Proposition~\ref{prop:atomic-endpoint} gives
\[
 d_H(F(r_1),F(r_2))
 \le L\sum_{\sigma\in[r_1,r_2)}[\Jump g(\sigma)]^\alpha.
\]
As $r_1,r_2\downarrow\tau$, the right-hand side is bounded by a tail of the convergent atomic series and tends to zero. Hence $F(r)$ is Hausdorff-Cauchy as $r\downarrow\tau$. Since $(X,d)$ is complete, the Hausdorff hyperspace $\K(X)$ is complete, and therefore $F(\tau+)$ exists.

For the selection $f$, Theorem~\ref{thm:main-preservation} and Theorem~\ref{thm:atomic-variation} give
\[
 d(f(r_1),f(r_2))
 \le\VH(F;[r_1,r_2])
 \le L\sum_{\sigma\in[r_1,r_2)}[\Jump g(\sigma)]^\alpha,
\]
so the same argument yields the existence of $f(\tau+)$.
\end{proof}

Define $J_F,J_f:[a,b)\to[0,\infty)$ by
\begin{equation}\label{eq:jumps-F-f}
 J_F(\tau)=d_H(F(\tau+),F(\tau)),
 \qquad
 J_f(\tau)=d(f(\tau+),f(\tau)).
\end{equation}

Define the variation functions $\Phi_F,\Phi_f:I\to[0,\infty)$ by
\[
 \Phi_F(t)=\VH(F;[a,t]),
 \qquad
 \Phi_f(t)=\Var(f;[a,t]).
\]

The variation functions $\Phi_F$ and $\Phi_f$ are left-continuous. Indeed, for $r<t$,
\[
 0\le\Phi_F(t)-\Phi_F(r)
 =\VH(F;[r,t])
 \le L\sum_{\tau\in[r,t)}[\Jump g(\tau)]^\alpha,
\]
and the tail on the right tends to zero as $r\uparrow t$. Since
\[
 0\le\Phi_f(t)-\Phi_f(r)=\Var(f;[r,t])\le\VH(F;[r,t]),
\]
the same conclusion holds for $\Phi_f$.

\begin{lemma}[Variation-jump identity]\label{lem:variation-jump}
For every $a\le\tau<b$,
\begin{equation}\label{eq:variation-jump-F}
 \Phi_F(\tau+)-\Phi_F(\tau)
 =J_F(\tau).
\end{equation}
Equivalently,
\[
 \lim_{r\downarrow\tau}\VH(F;[\tau,r])=J_F(\tau).
\]
\end{lemma}

\begin{proof}
The lower bound follows from
\[
 \VH(F;[\tau,r])\ge d_H(F(\tau),F(r))
\]
and $r\downarrow\tau$.

For the upper bound, take a partition $\tau=t_0<t_1<\cdots<t_n=r$. The first increment satisfies
\[
 d_H(F(\tau),F(t_1))
 \le J_F(\tau)+d_H(F(\tau+),F(t_1)).
\]
By passing to the right limit in Proposition~\ref{prop:atomic-endpoint},
\[
 d_H(F(\tau+),F(t_1))
 \le L\sum_{\sigma\in(\tau,t_1)}[\Jump g(\sigma)]^\alpha.
\]
Apply Proposition~\ref{prop:atomic-endpoint} to all later cells. Summing gives
\[
 \sum_{i=1}^n d_H(F(t_{i-1}),F(t_i))
 \le J_F(\tau)
 +L\sum_{\sigma\in(\tau,r)}[\Jump g(\sigma)]^\alpha.
\]
Take the supremum over partitions and then let $r\downarrow\tau$.
\end{proof}

\begin{remark}[Classical BV jump formulas]\label{rem:chistyakov-jumps}
The relation between one-sided jumps of a metric-valued BV map and the corresponding jumps of its variation function is well established in the literature; see, in particular, Chistyakov's jump formulas \cite{Chistyakov1998}. The additional $g$-H\"older assumption with exponent $\alpha>1$ removes the nonatomic contribution to variation and turns those local jump relations into the global identities below.
\end{remark}

The endpoint estimate also gives, for every $a\le\tau<b$, the pointwise jump bound
\begin{equation}\label{eq:set-jump-bound}
 J_F(\tau)\le L[\Jump g(\tau)]^\alpha.
\end{equation}

For the left-continuous increasing variation functions, define finite Borel measures by setting, for every $a\le s\le t\le b$,
\begin{equation}\label{eq:variation-measures}
 \nu_F([s,t))=\Phi_F(t)-\Phi_F(s)=\VH(F;[s,t]),
\end{equation}
\[
 \nu_f([s,t))=\Phi_f(t)-\Phi_f(s)=\Var(f;[s,t]).
\]
For convenience, recall the atomic measure introduced in Section~\ref{sec:atomic-variation},
\[
 \lambda_{\alpha,g}
 =\sum_{\tau\in[a,b)}[\Jump g(\tau)]^\alpha\delta_\tau,
\]
so that
\[
 \lambda_{\alpha,g}([s,t))
 =\sum_{\tau\in[s,t)}[\Jump g(\tau)]^\alpha.
\]
Theorem~\ref{thm:atomic-variation} gives $\nu_F([s,t))\le L\lambda_{\alpha,g}([s,t))$ on every half-open interval. By finite additivity on finite disjoint unions of such intervals and the standard monotone-class argument, this domination extends to all Borel sets; hence
\begin{equation}\label{eq:measure-dom-F}
 \nu_F\le L\lambda_{\alpha,g}.
\end{equation}
Thus $\nu_F$ is supported on the atoms of $\lambda_{\alpha,g}$, equivalently on the jump points of $g$. Lemma~\ref{lem:variation-jump} identifies the mass of $\nu_F$ at each such point with $J_F(\tau)$.

\begin{theorem}[Exact decomposition of variation into jumps]\label{thm:exact-jump}
Let $(X,d)$ be a complete metric space, let $g:I\to\mathbb R$ be left-continuous and nondecreasing, let $\alpha>1$, and let $F:I\to\K(X)$ be $g$-H\"older of exponent $\alpha$ with $L=[F]_{\alpha,g}<\infty$. Then, for every $a\le s<t\le b$,
\begin{equation}\label{eq:exact-F}
 \VH(F;[s,t])
 =\sum_{\tau\in[s,t)}J_F(\tau)
 =\sum_{\tau\in[s,t)}d_H(F(\tau+),F(\tau)).
\end{equation}
Moreover,
\begin{equation}\label{eq:exact-F-bound}
 \sum_{\tau\in[s,t)}J_F(\tau)
 \le L\sum_{\tau\in[s,t)}[\Jump g(\tau)]^\alpha.
\end{equation}
\end{theorem}

\begin{proof}
Equation \eqref{eq:measure-dom-F} shows that $\nu_F$ is supported on the atoms of $\lambda_{\alpha,g}$. Since the mass of $\nu_F$ at $\tau$ is $J_F(\tau)$, summing the atoms over $[s,t)$ gives \eqref{eq:exact-F}. Equation \eqref{eq:exact-F-bound} follows from \eqref{eq:set-jump-bound} or directly from \eqref{eq:measure-dom-F}.
\end{proof}

Since \eqref{eq:variation-preservation} gives $\nu_f([s,t))\le\nu_F([s,t))$ on every half-open interval, the same monotone-class argument yields
\begin{equation}\label{eq:measure-dom-f}
 \nu_f\le\nu_F.
\end{equation}
Thus $\nu_f$ is also purely atomic; that is, its mass is concentrated on the jump points. The single-valued counterpart of Lemma~\ref{lem:variation-jump} gives $\nu_f(\{\tau\})=J_f(\tau)$.

\begin{corollary}[Exact jump decomposition for the prescribed-point selection]\label{cor:exact-selection}
Fix $\theta\in I$ and $x_\theta\in F(\theta)$, and let $f:I\to X$ be the prescribed-point selection supplied by Theorem~\ref{thm:main-preservation}, with $f(\theta)=x_\theta$. Then, for every $a\le s<t\le b$,
\begin{equation}\label{eq:exact-f}
 \Var(f;[s,t])
 =\sum_{\tau\in[s,t)}J_f(\tau)
 =\sum_{\tau\in[s,t)}d(f(\tau+),f(\tau)).
\end{equation}
\end{corollary}

\begin{proof}
Since $\nu_f$ is purely atomic and $\nu_f(\{\tau\})=J_f(\tau)$, summing its atoms over $[s,t)$ and using $\nu_f([s,t))=\Var(f;[s,t])$ gives the result.
\end{proof}

\begin{corollary}[Pointwise preservation of jumps]\label{cor:jump-preservation}
For every $a\le\tau<b$,
\begin{equation}\label{eq:jump-by-jump}
 J_f(\tau)
 \le J_F(\tau)
 \le [F]_{\alpha,g}[\Jump g(\tau)]^\alpha.
\end{equation}
Consequently, if
\[
 \mathcal J(g)=\{\tau<b:\Jump g(\tau)>0\},
\]
\[
 \mathcal J(F)=\{\tau<b:F(\tau+)\ne F(\tau)\},
 \qquad
 \mathcal J(f)=\{\tau<b:f(\tau+)\ne f(\tau)\},
\]
then
\begin{equation}\label{eq:jump-hierarchy}
 \mathcal J(f)\subseteq\mathcal J(F)\subseteq\mathcal J(g).
\end{equation}
\end{corollary}

\begin{proof}
Apply \eqref{eq:measure-dom-f} to the singleton $\{\tau\}$ and then use \eqref{eq:set-jump-bound}.
\end{proof}

\begin{theorem}[Exact atomic formula for Riesz $p$-variation]\label{thm:exact-riesz-atomic}
Let $(X,d)$ be a complete metric space, let $g:I\to\mathbb R$ be left-continuous and nondecreasing, let $\alpha>1$, let $F:I\to\K(X)$ be $g$-H\"older of exponent $\alpha$, and let $p>1$. Then, for every $a\le s<t\le b$,
\[
 \mathcal R^H_{p,g}(F;[s,t])
 =
 \sum_{\substack{\tau\in[s,t)\\ \Jump g(\tau)>0}}
 \frac{J_F(\tau)^p}{[\Jump g(\tau)]^{p-1}}.
\]
The identity remains valid for countably many jump points, including families with accumulation points.
\end{theorem}

\begin{proof}
Fix a partition cell $[u,v]$. By Theorem~\ref{thm:exact-jump},
\[
 d_H(F(u),F(v))
 \le
 \sum_{\tau\in[u,v)}J_F(\tau),
\]
while
\[
 g(v)-g(u)
 \ge
 \sum_{\tau\in[u,v)}\Jump g(\tau).
\]
If the sum of clock jumps in the cell is zero, Theorem~\ref{thm:exact-jump} gives $d_H(F(u),F(v))=0$, so the cell has zero contribution. Otherwise, weighted H\"older gives, first for finite partial sums and then by monotone passage to the countable sum,
\[
 \frac{\left(\sum_{\substack{\tau\in[u,v)\\ \Jump g(\tau)>0}}J_F(\tau)\right)^p}
 {\left(\sum_{\substack{\tau\in[u,v)\\ \Jump g(\tau)>0}}\Jump g(\tau)\right)^{p-1}}
 \le
 \sum_{\substack{\tau\in[u,v)\\ \Jump g(\tau)>0}}
 \frac{J_F(\tau)^p}{[\Jump g(\tau)]^{p-1}}.
\]
Hence every partition sum is bounded above by the atomic series, which proves
\[
 \mathcal R^H_{p,g}(F;[s,t])
 \le
 \sum_{\substack{\tau\in[s,t)\\ \Jump g(\tau)>0}}
 \frac{J_F(\tau)^p}{[\Jump g(\tau)]^{p-1}}.
\]

For the reverse inequality, take any finite set of jump points
\[
 A=\{\tau_1,\ldots,\tau_N\}\subset[s,t).
\]
Choose pairwise disjoint right intervals $[\tau_j,r_j]$ that can be inserted into one partition. As $r_j\downarrow\tau_j$,
\[
 d_H(F(\tau_j),F(r_j))\to J_F(\tau_j),
 \qquad
 g(r_j)-g(\tau_j)\to\Jump g(\tau_j).
\]
The corresponding cells therefore give
\[
 \mathcal R^H_{p,g}(F;[s,t])
 \ge
 \sum_{j=1}^N
 \frac{J_F(\tau_j)^p}{[\Jump g(\tau_j)]^{p-1}}.
\]
Taking the supremum over all finite subsets $A$ yields the reverse inequality and the stated identity.
\end{proof}

\begin{corollary}[Atomic formula for the Riesz $p$-variation of the prescribed-point selection]\label{cor:atomic-riesz-selection}
Let $p>1$, fix $\theta\in I$ and $x_\theta\in F(\theta)$, and let $f:I\to X$ be the prescribed-point selection supplied by Theorem~\ref{thm:main-preservation}, with $f(\theta)=x_\theta$. Then, for every $a\le s<t\le b$,
\[
 \mathcal R_{p,g}(f;[s,t])
 =
 \sum_{\substack{\tau\in[s,t)\\ \Jump g(\tau)>0}}
 \frac{J_f(\tau)^p}{[\Jump g(\tau)]^{p-1}}
 \le
 \mathcal R^H_{p,g}(F;[s,t]).
\]
Thus, for $\alpha>1$, the same prescribed-point selection preserves Riesz $p$-variation through the pointwise bounds on its jumps.
\end{corollary}

\begin{proof}
By Theorem~\ref{thm:main-preservation}, the selection $f$ is $g$-H\"older of exponent $\alpha>1$, and Corollary~\ref{cor:exact-selection} gives its exact jump decomposition. Therefore the argument of Theorem~\ref{thm:exact-riesz-atomic} applies verbatim to $f$, yielding the displayed atomic identity with $J_f$ in place of $J_F$. Corollary~\ref{cor:jump-preservation} gives $J_f(\tau)\le J_F(\tau)$ for every jump point $a\le\tau<b$ of $g$, that is, whenever $\Jump g(\tau)>0$, which gives the final inequality.
\end{proof}

\begin{proposition}[Sharp exponent in the Riesz $p$-variation bound]\label{prop:riesz-exponent-law}
Let $p>1$, let $L=[F]_{\alpha,g}$, and set $q=1+p(\alpha-1)$. Then, for every $a\le s<t\le b$,
\[
 \mathcal R^H_{p,g}(F;[s,t])
 \le
 L^p
 \sum_{\tau\in[s,t)}[\Jump g(\tau)]^q.
\]
Both the exponent $q$ and the coefficient $L^p$ are attained in the one-jump class and therefore cannot be improved in general.
\end{proposition}

\begin{proof}
By \eqref{eq:set-jump-bound}, for every $a\le\tau<b$ with $\Jump g(\tau)>0$,
\[
 J_F(\tau)\le L[\Jump g(\tau)]^\alpha.
\]
Substitution into Theorem~\ref{thm:exact-riesz-atomic} gives
\[
 \frac{J_F(\tau)^p}{[\Jump g(\tau)]^{p-1}}
 \le
 L^p[\Jump g(\tau)]^{\alpha p-(p-1)}
 =L^p[\Jump g(\tau)]^{1+p(\alpha-1)}.
\]
To show sharpness, take a single clock jump of size $m>0$ and a Hausdorff jump $J_F=Lm^\alpha$. Then
\[
 \mathcal R^H_{p,g}=\frac{(Lm^\alpha)^p}{m^{p-1}}
 =L^pm^{1+p(\alpha-1)},
\]
so neither the exponent nor the coefficient can be reduced in general.
\end{proof}

\begin{remark}
The identity in Theorem~\ref{thm:exact-riesz-atomic} concerns Riesz $p$-variation measured with respect to the Stieltjes clock $g$. This is distinct from Wiener $p$-variation; in particular, the jump identity above does not extend in general to Wiener $p$-variation \cite{Chistyakov1998}.
\end{remark}

When $X$ is complete, the preceding results combine into the following chain of estimates for every $a\le s<t\le b$:
\[
 \begin{aligned}
 d(f(s),f(t))
 &\le\Var(f;[s,t])\\
 &\le\VH(F;[s,t])\\
 &=\sum_{\tau\in[s,t)}J_F(\tau)\\
 &\le[F]_{\alpha,g}\sum_{\tau\in[s,t)}[\Jump g(\tau)]^\alpha\\
 &\le[F]_{\alpha,g}|g(t)-g(s)|^\alpha.
 \end{aligned}
\]

\section{Explicit examples and sharpness}\label{sec:examples}

\begin{example}[A one-jump nonconvex map]\label{ex:one-jump}
Let $I=[0,1]$, $X=\mathbb R$, and define $g:I\to\mathbb R$ by
\[
 g(t)=
 \begin{cases}
 0,&0\le t\le\tfrac12,\\
 1,&\tfrac12<t\le1.
 \end{cases}
\]
Define $F:I\to\K(\mathbb R)$ by $F(t)=\{0,2\}$ when $g(t)=0$ and $F(t)=\{1,3\}$ when $g(t)=1$. Since $d_H(\{0,2\},\{1,3\})=1$, the $g$-H\"older seminorm of $F$ equals $1$ for every $\alpha>1$.

Now define the selection $f:I\to\mathbb R$ by
\[
 f(t)=
 \begin{cases}
 0,&0\le t\le\tfrac12,\\
 1,&\tfrac12<t\le1.
 \end{cases}
\]
Then the $g$-H\"older seminorm of $f$ is also $1$. Hence a single clock jump already permits a nonconstant compact-valued map that is $g$-H\"older of exponent $\alpha>1$, together with a selection having the same seminorm.
\end{example}

\begin{example}[Nonmonotone clock: regularity without local variation preservation]\label{ex:nonmonotone-variation-obstruction}
Let $X=\{a_1,b_1,c_1,a_2,b_2,c_2\}$ be equipped with the shortest-path metric of the six-cycle $a_1-b_1-c_1-a_2-b_2-c_2-a_1$, with every edge of length one. Set $A=\{a_1,a_2\}$, $B=\{b_1,b_2\}$, and $C=\{c_1,c_2\}$. Then $A,B,C\in\K(X)$ and $d_H(A,B)=d_H(B,C)=d_H(C,A)=1$.

Let $I=[0,3]$ and define $g:I\to\mathbb R$ by
\[
 g(t)=
 \begin{cases}
 0, & 0\le t<1,\\
 1, & 1\le t<2,\\
 2, & 2\le t<3,\\
 0, & t=3.
 \end{cases}
\]
Thus $g(I)=\{0,1,2\}$. Define $F:I\to\K(X)$ by $F(t)=A$ when $g(t)=0$, $F(t)=B$ when $g(t)=1$, and $F(t)=C$ when $g(t)=2$. For every $\alpha\ge1$, one has $[F]_{\alpha,g}=1$: the two unit clock gaps correspond to Hausdorff distance one, while $d_H(A,C)=1\le2^\alpha$.

Define $\widehat f:g(I)\to X$ by $\widehat f(0)=a_1$, $\widehat f(1)=b_1$, and $\widehat f(2)=c_1$, and set $f=\widehat f\circ g:I\to X$. Since $d(a_1,b_1)=d(b_1,c_1)=1$ and $d(a_1,c_1)=2\le2^\alpha$, $f$ is a selection of $F$ with
\[
 [f]_{\alpha,g}=1=[F]_{\alpha,g}.
\]
Thus the regularity bound is attained with equality.

No selection $f:I\to X$ with $[f]_{\alpha,g}\le1$ can also satisfy $\Var(f;[s,t])\le\VH(F;[s,t])$ on every subinterval $[s,t]\subset I$. Such a selection is constant on every level set of $g$ and is therefore determined by points $x_A\in A$, $x_B\in B$, and $x_C\in C$. The two unit clock gaps require $d(x_A,x_B)\le1$ and $d(x_B,x_C)\le1$. In the six-cycle metric, these inequalities force the same index throughout: either $x_A=a_1$, $x_B=b_1$, $x_C=c_1$, or $x_A=a_2$, $x_B=b_2$, $x_C=c_2$. Hence $d(x_C,x_A)=2$.

On the subinterval $[2,3]$, however, $\VH(F;[2,3])=d_H(C,A)=1$, whereas $f$ is constant with value $x_C$ on $[2,3)$ and takes the value $x_A$ at $t=3$. Therefore
\[
 \Var(f;[2,3])=d(x_C,x_A)=2>1=\VH(F;[2,3]).
\]
Thus preservation of the $g$-H\"older bound does not, without monotonicity of the clock, guarantee preservation of the local Hausdorff-variation bound. Here the obstruction is created by the return of $g$ from level $2$ to the previously attained value $0$.
\end{example}

For comparison, when $g(t)=t$, Corollary~\ref{cor:continuous-rigidity} implies that every compact-valued $g$-H\"older map of exponent $\alpha>1$ is constant, whereas Example~\ref{ex:alpha-one} shows that nontrivial continuous variation is possible at the boundary $\alpha=1$.

Recall the cumulative function
\[
 A_{\alpha,g}(t)=\lambda_{\alpha,g}([a,t))=\sum_{\tau\in[a,t)}[\Jump g(\tau)]^\alpha.
\]
The following two-point construction shows that the regularity, variation, and jump bounds obtained above are all attained by the same example.

\begin{proposition}[Simultaneous attainment of the regularity, variation, and jump bounds]\label{prop:sharpness}
Let $g:I=[a,b]\to\mathbb R$ be a left-continuous pure-jump nondecreasing clock with at least one positive jump, and let $\alpha>1$ and $L,R>0$. Let $A_{\alpha,g}$ be the cumulative function displayed above, and define $F:I\to\K(\mathbb R)$ and $f:I\to\mathbb R$ by
\begin{equation}\label{eq:sharp-family}
 F(t)=L A_{\alpha,g}(t)+\{0,R\},
 \qquad
 f(t)=L A_{\alpha,g}(t).
\end{equation}
Then $f$ is a selection of $F$ and
\begin{equation}\label{eq:sharp-constants}
 [F]_{\alpha,g}=[f]_{\alpha,g}=L.
\end{equation}
Moreover, for every $a\le s<t\le b$,
\begin{equation}\label{eq:sharp-variation}
 d(f(s),f(t))
 =\Var(f;[s,t])
 =\VH(F;[s,t])
 =L\bigl(A_{\alpha,g}(t)-A_{\alpha,g}(s)\bigr).
\end{equation}
At every $a\le\tau<b$ with $\Jump g(\tau)>0$,
\begin{equation}\label{eq:sharp-jumps}
 J_f(\tau)=J_F(\tau)=L[\Jump g(\tau)]^\alpha.
\end{equation}
Thus the regularity, local variation, and jump estimates obtained above are all sharp for this example.
\end{proposition}

\begin{proof}
Fix $a\le s<t\le b$. Translations of the same compact set have Hausdorff distance equal to the translation size, so
\[
 d_H(F(s),F(t))
 =L\bigl(A_{\alpha,g}(t)-A_{\alpha,g}(s)\bigr)
 =L\sum_{\tau\in[s,t)}[\Jump g(\tau)]^\alpha.
\]
Since $g$ is pure-jump and nondecreasing,
\[
 g(t)-g(s)=\sum_{\tau\in[s,t)}\Jump g(\tau),
\]
and therefore
\[
 d_H(F(s),F(t))
 \le L[g(t)-g(s)]^\alpha.
\]
Hence $[F]_{\alpha,g}\le L$. Across any positive jump $\tau$, the defining quotient approaches $L$ from the right, so the optimal constant equals $L$. The same argument applies to $f$.

The function $A_{\alpha,g}$ is nondecreasing. Hence the variation of $f$ on $[s,t]$ equals $L\bigl(A_{\alpha,g}(t)-A_{\alpha,g}(s)\bigr)$. The Hausdorff increments of $F$ telescope in the same way, proving \eqref{eq:sharp-variation}. Equation~\eqref{eq:sharp-jumps} follows directly from the jump of $A_{\alpha,g}$ at each jump point $\tau$ of $g$.
\end{proof}

\begin{remark}
Proposition~\ref{prop:sharpness} shows simultaneously that the constants in \eqref{eq:variation-preservation}, \eqref{eq:seminorm-preservation}, and \eqref{eq:jump-by-jump} are sharp in general. The example remains nonconvex because every value consists of two points. If the prescribed point is the lower or upper branch at a fixed time, the corresponding translated branch gives a prescribed-point selection for which the same bounds are attained.
\end{remark}

\section{Convex \texorpdfstring{$g$}{g}-H\"older selections below the Lipschitz threshold}\label{sec:convex-sub}

This section considers the range $0<\alpha<1$ for compact convex values in Euclidean space. The classical Steiner point map preserves the H\"older exponent under the Hausdorff metric. We additionally require the selection to pass through a prescribed graph point. For the selection results below, $g:I\to\mathbb R$ need not be monotone. Throughout this section, $\mathcal K_c(\mathbb R^n)$ denotes the family of nonempty compact convex subsets of $\mathbb R^n$.

Set
\[
 \Lambda_n:=\frac{2\,\Gamma\!\left(\frac n2+1\right)}{\sqrt\pi\,\Gamma\!\left(\frac{n+1}{2}\right)}.
\]
Here, $\Gamma$ denotes the Euler gamma function. For the classical Steiner point $\operatorname{st}(K)$ of $K\in\mathcal K_c(\mathbb R^n)$ one has
\[
 \operatorname{st}(K)\in K,
 \qquad
 \|\operatorname{st}(A)-\operatorname{st}(B)\|
 \le \Lambda_n d_H(A,B);
\]
see, for example, \cite{ArutyunovObukhovskii2017,DeutschLiPark1989}. Generalized Steiner selections preserving H\"older order are also developed by Dentcheva \cite{Dentcheva2000}.

\begin{proposition}[Steiner selection preserving the $g$-H\"older exponent]\label{prop:convex-steiner}
Let $0<\alpha<1$, let $g:I\to\mathbb R$, and let $F:I\to\mathcal K_c(\mathbb R^n)$ be $g$-H\"older. Then $F$ admits a selection $f:I\to\mathbb R^n$ with
\[
 [f]_{\alpha,g}\le \Lambda_n [F]_{\alpha,g}.
\]
In particular, the H\"older exponent is preserved.
\end{proposition}

\begin{proof}
Let $L=[F]_{\alpha,g}$ and define $f(t)=\operatorname{st}(F(t))$. Then $f(t)\in F(t)$ and, for all $s,t\in I$,
\[
 \|f(t)-f(s)\|
 \le \Lambda_n d_H(F(t),F(s))
 \le \Lambda_n L|g(t)-g(s)|^\alpha.
\]
\end{proof}

Proposition~\ref{prop:convex-steiner} does not prescribe the value of the selection at a chosen time. This raises the prescribed-point question: if $\theta\in I$ and $x_\theta\in F(\theta)$ are fixed, can one impose $f(\theta)=x_\theta$ while retaining the original constant $[F]_{\alpha,g}$? The answer depends on the dimension.

\begin{proposition}[Same-constant prescribed-point selection in one dimension]\label{prop:scalar-subholder}
Let $0<\alpha<1$, let $g:I\to\mathbb R$, and let $F:I\to\mathcal K_c(\mathbb R)$ be $g$-H\"older. For every $\theta\in I$ and every $x_\theta\in F(\theta)$, there exists a selection $f:I\to\mathbb R$ such that
\[
 f(\theta)=x_\theta,
 \qquad
 [f]_{\alpha,g}\le [F]_{\alpha,g}.
\]
\end{proposition}

\begin{proof}
Write $F(t)=[a(t),b(t)]$ and let $x_\theta\in F(\theta)$ be fixed. Define
\[
 f(t)=\min\{b(t),\max\{a(t),x_\theta\}\},
\]
that is, $f(t)$ is the metric projection of the fixed point $x_\theta$ onto the interval $F(t)$. Then $f(t)\in F(t)$ and $f(\theta)=x_\theta$. For compact intervals $A,B\subset\mathbb R$ and every $x\in\mathbb R$,
\[
 |P_A(x)-P_B(x)|\le d_H(A,B).
\]
Hence, with $L=[F]_{\alpha,g}$,
\[
 |f(t)-f(s)|\le d_H(F(t),F(s))
 \le L|g(t)-g(s)|^\alpha.
\]
\end{proof}

The same-constant result above is specific to one dimension. Already in the plane, three compact convex values are enough to destroy the same-constant property under a prescribed-point constraint.

\begin{theorem}[Higher-dimensional obstruction to same-constant prescribed-point selection]\label{thm:higher-dimensional-obstruction}
For every $0<\alpha<1$ and every $n\ge2$, there exist a left-continuous nondecreasing map $g:I\to\mathbb R$, a map $F:I\to\mathcal K_c(\mathbb R^n)$ with $[F]_{\alpha,g}=1$, a point $\theta\in I$, and a point $x_\theta\in F(\theta)$ such that no selection $f:I\to\mathbb R^n$ satisfying $f(\theta)=x_\theta$ can satisfy $[f]_{\alpha,g}\le1$.
\end{theorem}

The obstruction already occurs with three convex sets in the plane. The construction uses three clock levels: the prescribed value forces the middle selection to a unique point, while the constraints from the first and third levels make the final H\"older bound incompatible with the H\"older constraint between the middle and final levels. Figure~\ref{fig:subholder-dimension} summarizes the geometric contrast; the exact construction is given in Appendix~\ref{app:planar-obstruction}.

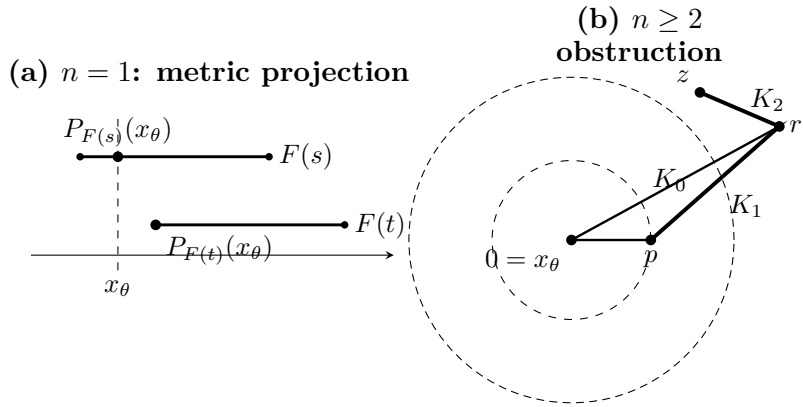
\begin{figure}[H]
\centering
\begin{tikzpicture}[x=1cm,y=1cm,>=stealth,font=\small]
  \begin{scope}
    \node[font=\bfseries] at (2.45,2.25) {(a) $n=1$: metric projection};
    \draw[->] (0.1,-0.15) -- (4.9,-0.15);
    \draw[line width=1.1pt] (0.75,1.15) -- (3.25,1.15);
    \fill (0.75,1.15) circle (1.4pt);
    \fill (3.25,1.15) circle (1.4pt);
    \node[right] at (3.25,1.15) {$F(s)$};
    \draw[line width=1.1pt] (1.75,0.25) -- (4.25,0.25);
    \fill (1.75,0.25) circle (1.4pt);
    \fill (4.25,0.25) circle (1.4pt);
    \node[right] at (4.25,0.25) {$F(t)$};
    \draw[dashed] (1.25,-0.35) -- (1.25,1.65);
    \node[below] at (1.25,-0.35) {$x_\theta$};
    \fill (1.25,1.15) circle (2pt);
    \node[above] at (1.25,1.15) {$P_{F(s)}(x_\theta)$};
    \fill (1.75,0.25) circle (2pt);
    \node[below right] at (1.75,0.25) {$P_{F(t)}(x_\theta)$};
  \end{scope}

  \begin{scope}[xshift=7.0cm]
    \node[font=\bfseries,align=center] at (1.15,2.78) {(b) $n\ge2$\\[-1pt] obstruction};
    \coordinate (O) at (0.25,0.05);
    \coordinate (P) at (1.30,0.05);
    \coordinate (R) at (3.00,1.55);
    \coordinate (Z) at (1.95,2.00);
    \draw[dashed] (O) circle (1.05);
    \draw[densely dashed] (O) circle (2.15);
    \draw[line width=0.9pt] (O) -- (P) -- (R) -- cycle;
    \draw[line width=1.5pt] (P) -- (R);
    \draw[line width=1.5pt] (R) -- (Z);
    \fill (O) circle (2pt);
    \fill (P) circle (2pt);
    \fill (R) circle (2pt);
    \fill (Z) circle (2pt);
    \node[below left] at (O) {$0=x_\theta$};
    \node[below] at (P) {$p$};
    \node[right] at (R) {$r$};
    \node[above left] at (Z) {$z$};
    \node at (1.55,0.85) {$K_0$};
    \node[below right] at (2.20,0.78) {$K_1$};
    \node[right] at (2.48,1.90) {$K_2$};
  \end{scope}
\end{tikzpicture}
\caption{Geometric contrast for prescribed-point selection below the Lipschitz threshold. In one dimension, metric projection of a fixed prescribed value onto moving intervals is nonexpansive with respect to the Hausdorff metric. In the planar example, the prescribed origin forces the middle choice to $p$, while the third level creates incompatible distance constraints. The right panel is schematic; the exact coordinates are given in Appendix~\ref{app:planar-obstruction}.}
\label{fig:subholder-dimension}
\end{figure}

\begin{proof}
See Appendix~\ref{app:planar-obstruction}, \emph{Proof of the higher-dimensional obstruction to same-constant prescribed-point selection}.
\end{proof}

Theorem~\ref{thm:higher-dimensional-obstruction} shows that the same-constant prescribed-point property cannot hold in general from dimension two onward. The next result gives an explicit bound.

\begin{theorem}[Prescribed-point selection with an explicit $g$-H\"older bound]\label{thm:subholder-prescribed}
Let $0<\alpha<1$, let $g:I\to\mathbb R$, and let $F:I\to\mathcal K_c(\mathbb R^n)$ be $g$-H\"older. Then, for every $\theta\in I$ and every $x_\theta\in F(\theta)$, there exists a selection $f:I\to\mathbb R^n$ satisfying
\[
 f(\theta)=x_\theta,
 \qquad
 [f]_{\alpha,g}\le 9\Lambda_n[F]_{\alpha,g}.
\]
\end{theorem}

\begin{proof}
Set $L=[F]_{\alpha,g}$ and define
\[
 R(t)=3L|g(t)-g(\theta)|^\alpha,
 \qquad
 G(t)=F(t)\cap\overline B(x_\theta,R(t)).
\]
The set $G(t)$ is nonempty, compact, and convex. Indeed,
\[
 d(x_\theta,F(t))
 \le d_H(F(\theta),F(t))
 \le L|g(t)-g(\theta)|^\alpha
 =\frac{R(t)}{3}.
\]
Moreover, $G(\theta)=\{x_\theta\}$.

Fix $s,t\in I$ and put
\[
 \delta=L|g(t)-g(s)|^\alpha.
\]
Since $0<\alpha<1$,
\[
 \bigl||u|^\alpha-|v|^\alpha\bigr|\le|u-v|^\alpha,
\]
and hence $|R(t)-R(s)|\le3\delta$. Let $y\in G(s)$. Choose $z\in F(t)$ with $\|y-z\|\le\delta$. Then
\[
 \|z-x_\theta\|
 \le R(t)+4\delta.
\]
If $z\in\overline B(x_\theta,R(t))$, then $z\in G(t)$ and $d(y,G(t))\le\delta$. Otherwise put $r=\|z-x_\theta\|>R(t)$. Choose $p\in F(t)$ with
\[
 q:=\|p-x_\theta\|\le R(t)/3.
\]
Along the segment $[z,p]\subset F(t)$ choose
\[
 \lambda=\frac{r-R(t)}{r-q}
 \qquad\text{and}\qquad
 w=(1-\lambda)z+\lambda p.
\]
By convexity of the norm,
\[
 \|w-x_\theta\|
 \le(1-\lambda)r+\lambda q
 =R(t),
\]
so $w\in G(t)$. Furthermore,
\[
 \|z-w\|
 =\lambda\|z-p\|
 \le (r-R(t))\frac{r+q}{r-q}
 \le2(r-R(t))
 \le8\delta,
\]
where the ratio is at most $2$ because $q\le R(t)/3<r/3$. Consequently,
\[
 d(y,G(t))\le\|y-w\|\le9\delta.
\]
Interchanging $s$ and $t$ gives
\[
 d_H(G(s),G(t))\le9L|g(t)-g(s)|^\alpha.
\]
Applying the Steiner point to $G$ and using $G(\theta)=\{x_\theta\}$ yields a selection $f(t)=\operatorname{st}(G(t))$ satisfying $f(\theta)=x_\theta$ and
\[
 [f]_{\alpha,g}\le9\Lambda_nL.
\]
\end{proof}

The prescribed-point condition therefore behaves differently in one and higher dimensions below the Lipschitz threshold: the original $g$-H\"older constant can be retained in one dimension, while the same-constant property may fail from dimension two onward; a selection with the same exponent remains available with an explicit enlarged constant.

\begin{remark}[On the coefficient $9$]\label{rem:subholder-nine}
The coefficient $9$ is an explicit upper bound produced by the radial truncation used in the proof. More generally, replacing the radius factor $3$ by $c>1$ leads, through the same estimate, to the coefficient
\[
 1+\frac{(c+1)^2}{c-1}.
\]
This particular upper bound is minimized at $c=3$, where it equals $9$. No claim is made that $9\Lambda_n$ is the optimal universal prescribed-point constant.
\end{remark}

\section{\texorpdfstring{$g$}{g}-continuous selections and the topology of \texorpdfstring{$g(I)$}{g(I)}}\label{sec:g-continuous}

Throughout the main part of this section, let $g:I\to\mathbb R$. We consider $g$-continuous selections without a quantitative H\"older bound. The problem depends only on the image $g(I)$, and no monotonicity assumption on $g$ is needed for the factorization and characterization results below. Additional Stieltjes assumptions will be imposed only when pure-jump clocks are considered.

\begin{proposition}[Factorization through $g(I)$]\label{prop:g-factorization}
A set-valued map $F:I\to\K(X)$ is $g$-continuous if and only if there exists a Hausdorff-continuous map
\[
 \widehat F:g(I)\to\K(X)
\]
such that $F(t)=\widehat F(g(t))$ for every $t\in I$. Likewise, a single-valued map $f:I\to X$ is $g$-continuous if and only if $f=\widehat f\circ g$ for some continuous $\widehat f:g(I)\to X$.
\end{proposition}

\begin{proof}
For $s,t\in I$, if $g(s)=g(t)$, then $g$-continuity forces $F(s)=F(t)$, so $\widehat F(g(t))=F(t)$ is well defined. The definitions of continuity are then identical after passage to the variable $y=g(t)$. The converse follows directly from the continuity of $\widehat F$ on $g(I)$.
\end{proof}

Recall that a topological space is called zero-dimensional if it has a base consisting of sets that are both open and closed.

The argument uses the classical zero-dimensional Michael selection theorem: every lower semicontinuous closed-valued map from a zero-dimensional paracompact space into a complete metric space has a continuous selection \cite{Michael1956}. Since $g(I)$ is a subspace of $\mathbb R$, it is metrizable and hence paracompact.

\begin{theorem}[Prescribed-point $g$-continuous selection]\label{thm:g-continuous-positive}
Let $g:I\to\mathbb R$, let $(X,d)$ be a complete metric space, and suppose that $g(I)$ is zero-dimensional. If $F:I\to\K(X)$ is Hausdorff $g$-continuous, then for every $\theta\in I$ and every $x_\theta\in F(\theta)$ there exists a $g$-continuous selection $f:I\to X$ satisfying $f(\theta)=x_\theta$.
\end{theorem}

\begin{proof}
Let $y_\theta=g(\theta)\in g(I)$ and let $\widehat F:g(I)\to\K(X)$ be the factor map from Proposition~\ref{prop:g-factorization}. Define $\Psi:g(I)\to\K(X)$ by
\[
 \Psi(y)=
 \begin{cases}
 \{x_\theta\},&y=y_\theta,\\
 \widehat F(y),&y\ne y_\theta.
 \end{cases}
\]
The map $\Psi$ is closed-valued. It is lower semicontinuous away from $y_\theta$ because it agrees locally with $\widehat F$. At $y_\theta$, let $U\subset X$ be open with $x_\theta\in U$. Choose $\varepsilon>0$ with $B(x_\theta,\varepsilon)\subset U$. Hausdorff continuity of $\widehat F$ implies that for $y$ near $y_\theta$, $\widehat F(y)$ contains a point within $\varepsilon$ of $x_\theta$, hence $\Psi(y)\cap U\ne\varnothing$. Thus $\Psi$ is lower semicontinuous.

Michael's theorem gives a continuous selection $\widehat f:g(I)\to X$ of $\Psi$. Since $\Psi(y_\theta)=\{x_\theta\}$, $\widehat f(y_\theta)=x_\theta$. Define $f:I\to X$ by $f(t)=\widehat f(g(t))$. Then $f(t)\in F(t)$ for every $t\in I$, $f(\theta)=x_\theta$, and $f$ is $g$-continuous by Proposition~\ref{prop:g-factorization}.
\end{proof}

For subspaces of the real line, zero-dimensionality is equivalent to total disconnectedness \cite{Engelking1995}. Since the connected subsets of $\mathbb R$ are intervals, a subset $D\subset\mathbb R$ is zero-dimensional if and only if it contains no nondegenerate interval.

Bressan and Wang recall an explicit Hausdorff-continuous compact nonconvex multifunction $G:[0,1]\to\K(\mathbb R^2)$ with no continuous selection \cite{BressanWang2009}. This supplies the converse direction below.

\begin{theorem}[Zero-dimensionality characterization for $g$-continuous selections]\label{thm:clock-image-characterization}
Let $g:I\to\mathbb R$. The following are equivalent:
\begin{enumerate}
\renewcommand{\labelenumi}{(\roman{enumi})}
\item $g(I)$ is zero-dimensional;
\item for every complete metric space $X$, every compact-valued Hausdorff $g$-continuous map $F:I\to\K(X)$, and every $\theta\in I$ and $x_\theta\in F(\theta)$, there exists a $g$-continuous selection $f:I\to X$ satisfying $f(\theta)=x_\theta$.
\end{enumerate}
\end{theorem}

\begin{proof}
The implication (i)$\Rightarrow$(ii) is Theorem~\ref{thm:g-continuous-positive}.

Suppose (i) fails. By the preceding classical characterization, $g(I)$ contains a nondegenerate interval $J=[u,v]$, where $u,v\in g(I)$ and $u<v$. Let $G:[0,1]\to\K(\mathbb R^2)$ be a Hausdorff-continuous compact-valued multifunction with no continuous selection as in \cite{BressanWang2009}. Define $\rho:g(I)\to[0,1]$ by
\[
 \rho(y)=
 \begin{cases}
 0,&y\le u,\\
 \dfrac{y-u}{v-u},&u\le y\le v,\\
 1,&y\ge v.
 \end{cases}
\]
Define $\widehat F:g(I)\to\K(\mathbb R^2)$ by $\widehat F(y)=G(\rho(y))$. Then $\widehat F$ is Hausdorff-continuous. A continuous selection of $\widehat F$ would restrict to a continuous selection of $G$ on $J$ after affine reparametrization, a contradiction. Define the pullback $F:I\to\K(\mathbb R^2)$ by $F(t)=\widehat F(g(t))$; it is $g$-continuous by Proposition~\ref{prop:g-factorization}. If $f:I\to\mathbb R^2$ were a $g$-continuous selection of $F$, Proposition~\ref{prop:g-factorization} would give a continuous map $\widehat f:g(I)\to\mathbb R^2$ with $f(t)=\widehat f(g(t))$. Since $g:I\to g(I)$ is surjective, the selection property implies $\widehat f(y)\in\widehat F(y)$ for every $y\in g(I)$, contradicting the choice of $\widehat F$. Hence (ii) fails.
\end{proof}

Thus, at the level of $g$-continuity, once $g(I)$ contains a nondegenerate interval, the classical nonconvex selection failure reappears.

\begin{corollary}[Pure-jump Stieltjes clocks]\label{cor:pure-jump-zero-dim}
Let $g:I\to\mathbb R$ be a left-continuous nondecreasing pure-jump Stieltjes clock, so that $\mu_g$ is purely atomic and
\begin{equation}\label{eq:pure-jump-definition}
 g(t)-g(a)=\sum_{\tau\in[a,t)}\Jump g(\tau),\qquad t\in I.
\end{equation}
Then $g(I)$ is zero-dimensional. Consequently, for every complete metric space $(X,d)$, every Hausdorff $g$-continuous map $F:I\to\K(X)$, and every $\theta\in I$ and $x_\theta\in F(\theta)$, there exists a $g$-continuous selection $f:I\to X$ such that $f(\theta)=x_\theta$.
\end{corollary}

\begin{proof}
For every $a\le\tau<b$ with $\Jump g(\tau)>0$, consider the open jump gap $(g(\tau),g(\tau+))$. These jump gaps are pairwise disjoint and do not meet $g(I)$. By \eqref{eq:pure-jump-definition}, their total length equals $g(b)-g(a)$. Hence $g(I)$ has Lebesgue measure zero and therefore cannot contain a nondegenerate interval. The preceding classical characterization and Theorem~\ref{thm:clock-image-characterization} now give the result.
\end{proof}

\begin{example}[A zero-dimensional image for a Stieltjes clock that is not pure-jump]\label{ex:svc-clock}
Let $C\subset[0,1]$ be the standard Smith--Volterra--Cantor set, which is closed, nowhere dense, contains $0$ and $1$, and has Lebesgue measure $1/2$. Define
\[
 g(t)=\min\bigl(C\cap[t,1]\bigr),\qquad 0\le t\le1.
\]
Then $g$ is nondecreasing and left-continuous, and $g(I)=C$. If $(a_k,b_k)$ are the complementary gaps of $C$, then $\Jump g(a_k)=b_k-a_k$, so
\[
 \sum_k \Jump g(a_k)=\frac12<1=g(1)-g(0).
\]
Thus $g$ is not pure-jump, while $g(I)$ is zero-dimensional because it contains no nondegenerate interval. The zero-dimensionality characterization therefore applies even though the Stieltjes measure of $g$ is not purely atomic.
\end{example}

\section{Discussion and conclusions}\label{sec:discussion}

At and above the Lipschitz threshold, $g$-H\"older regularity and local variation require different assumptions for their preservation. For $\alpha\ge1$, a prescribed-point selection can preserve $g$-H\"older regularity without assuming monotonicity of $g$, and its $g$-H\"older seminorm does not increase. This alone does not ensure the corresponding variation bound on every subinterval, as shown by the nonmonotone example in Section~\ref{sec:examples}. When $g$ is nondecreasing, one prescribed-point selection preserves both $g$-H\"older regularity and the local Hausdorff-variation control. Thus monotonicity is not needed for regularity preservation, while the nondecreasing assumption provides the structure used to obtain local variation control for the same selection.

A further change occurs when the exponent exceeds one. At $\alpha=1$, continuous parts of the clock may still contribute to variation. For $\alpha>1$, however, the continuous part of a left-continuous nondecreasing Stieltjes clock contributes no Hausdorff variation: the variation of the set-valued map is exactly the sum of its jumps. Nonconstant behavior is therefore still possible, but it is carried by clock jumps; Example~\ref{ex:one-jump} shows that a single jump already suffices. The prescribed-point selection has the corresponding jump decomposition, and each of its jumps is bounded by the jump of the set-valued map at the same point. Table~\ref{tab:regimes} summarizes this change across the Lipschitz threshold.

Below the Lipschitz threshold, we work with compact convex Euclidean-valued maps, in accordance with the convexity assumptions used in classical H\"older selection results. For $0<\alpha<1$, a $g$-H\"older map $F:I\to\mathcal K_c(\mathbb R^n)$ admits a prescribed-point selection with the same H\"older exponent. In one dimension, the original $g$-H\"older constant can also be retained. From dimension two onward, this same-constant property may fail, although a prescribed-point selection with the same exponent remains available with an explicit controlled enlargement of the constant. Figure~\ref{fig:subholder-dimension} illustrates the geometric source of this distinction.

The local variation estimate has a further consequence for Riesz $p$-variation. The same prescribed-point selection preserves every finite Riesz $p$-variation considered with respect to a nondecreasing external Stieltjes clock. Thus the selection obtained from the local variation estimate remains valid for a broader family of variation functionals.

A different situation appears when only $g$-continuity is assumed. Since the set-valued map factors through $g(I)$, the existence of prescribed-point $g$-continuous selections is determined by the topology of this set. For complete metric targets, the relevant condition is the zero-dimensionality of $g(I)$. Pure-jump Stieltjes clocks satisfy this condition, but the non-pure-jump example shows that zero-dimensionality also occurs beyond clocks whose Stieltjes measure is purely atomic.

Taken together, the results distinguish the roles of the geometry of the values and the structure of the clock. Below the Lipschitz threshold, compact convex values allow the H\"older exponent to be preserved under the prescribed-point constraint, while preservation of the original constant depends on the dimension. At and above the threshold, $g$-H\"older regularity can be preserved without assuming monotonicity of $g$, whereas local variation control for the same selection is obtained for nondecreasing clocks. For exponents greater than one, this variation is carried entirely by jumps. At the level of $g$-continuity, the corresponding prescribed-point selection property is characterized by the zero-dimensionality of $g(I)$.

\appendix

\section{Proof of the small-mass partition lemma}\label{app:small-mass-proof}

\begin{proof}[Proof of Lemma~\ref{lem:small-mass}]
Define
\[
 H(x)=\mu([u,x)),\qquad u\le x\le v,
\]
and put $M=H(v)$. If $M\le2\delta$, the one-cell partition is sufficient.

Assume therefore that $M>2\delta$. For every integer $k\ge1$ such that $k\delta<M$, define
\[
 q_k=\inf\{x\in[u,v):H(x+)\ge k\delta\},
 \qquad H(x+)=H(x)+\mu(\{x\}).
\]
The left continuity of $H$, together with the right continuity of $x\mapsto H(x+)$ and the bound $\mu(\{x\})\le\delta$, gives
\[
 (k-1)\delta\le H(q_k)\le k\delta.
\]

Now remove repetitions from the finite family consisting of $u$, the points $q_k$, and $v$, and write the remaining points in increasing order as
\[
 u=r_0<r_1<\cdots<r_N=v.
\]
We claim that every consecutive pair satisfies $H(r_j)-H(r_{j-1})\le2\delta$.
Indeed, suppose first that $r_{j-1}$ and $r_j$ are interior quantile points. Let $q_k=r_{j-1}$ correspond to the largest index represented at $r_{j-1}$. Since $r_j$ is the next distinct quantile point, it corresponds to $q_{k+1}$. Hence $H(r_j)\le(k+1)\delta$ and $H(r_{j-1})\ge(k-1)\delta$, which proves the claimed bound.
The same bound holds for the first cell. If $q_1>u$, its mass is at most $\delta$; if $q_1=u$, the next distinct quantile point is reached no later than the second $\delta$-level, so the first cell has mass at most $2\delta$.

For the last cell, let $K$ be the largest integer satisfying $K\delta<M$. Then $M\le(K+1)\delta$, while the last quantile point satisfies $H(r_{N-1})\ge(K-1)\delta$, and therefore $M-H(r_{N-1})\le2\delta$.
Since
\[
 H(r_j)-H(r_{j-1})=\mu([r_{j-1},r_j)),
\]
we obtain \eqref{eq:small-mass-cells} for every $j=1,\ldots,N$.
\end{proof}

\section{Proof of the higher-dimensional obstruction to same-constant prescribed-point selection}\label{app:planar-obstruction}

\begin{proof}[Proof of Theorem~\ref{thm:higher-dimensional-obstruction}]
It is enough to construct the example in $\mathbb R^2$, since it can then be embedded isometrically into $\mathbb R^n$. Let $I=[0,3]$ and
\[
 g(t)=
 \begin{cases}
  0, & 0\le t\le1,\\
  \frac12, & 1<t\le2,\\
  1, & 2<t\le3.
 \end{cases}
\]
This clock is nondecreasing and left-continuous. Put $a=2^{-\alpha}$, so $1/2<a<1$. Choose $c$ such that
\[
 \max\left\{0,\frac{1}{2a^2}-1\right\}<c<\frac1a-1,
\]
which is possible because $a>1/2$, and set $s=\sqrt{1-c^2}$. Define
\[
 p=a(1,0),
 \qquad
 d=a(c,s),
 \qquad
 r=p+d,
 \qquad
 w=(-s,c),
 \qquad
 z=r+sa\,w.
\]
Then $d\perp w$,
\[
 \|r\|^2=2a^2(1+c)>1,
 \qquad
 \|z\|^2=a^2(1+c)^2<1.
\]
Now set
\[
 K_0=\operatorname{conv}\{0,p,r\},
 \qquad
 K_1=[p,r],
 \qquad
 K_2=[r,z],
\]
and let $F$ take the values $K_0,K_1,K_2$ on the three level sets of $g$, respectively.

Since $K_1\subset K_0$ and
\[
 \|p+\tau d\|^2=a^2(1+2c\tau+\tau^2),
 \qquad 0\le\tau\le1,
\]
the point $p$ is the unique point of $K_1$ at distance at most $a$ from the origin and $d_H(K_0,K_1)=a$.
The segments $K_1$ and $K_2$ meet at $r$, with lengths $a$ and $sa<a$, respectively, so $d_H(K_1,K_2)\le a$. Finally, the distance to a nonempty closed convex set is convex. At the vertices $0,p,r$ of $K_0$ the distances to $K_2$ are strictly less than $1$, because $\|z\|<1$, $\|p-r\|=a<1$, and $r\in K_2$. Hence every point of $K_0$ has distance less than $1$ from $K_2$. Conversely, the endpoint distances from $K_2=[r,z]$ to the convex set $K_0$ are $0$ and at most $\|z-r\|=sa<1$, so every point of $K_2$ has distance less than $1$ from $K_0$. Thus
\[
 d_H(K_0,K_2)<1.
\]
Because the clock increments are $1/2,1/2,1$, these estimates give $[F]_{\alpha,g}=1$.

Prescribe $x_\theta=0$ at any $\theta\in[0,1]$. If a selection $f$ through this point satisfied $[f]_{\alpha,g}\le1$, its value $x_1\in K_1$ on the middle level would obey $\|x_1\|\le a$. By the uniqueness above, $x_1=p$. Its value on the last level has the form
\[
 x_2=r+\mu a w,
 \qquad 0\le\mu\le s.
\]
The first-to-last constraint gives $\|x_2\|\le1$. Since $\|r\|>1$, necessarily $\mu>0$. Using $r-p=d$ and $d\perp w$,
\[
 \|x_2-p\|^2
 =\|d+\mu aw\|^2
 =a^2(1+\mu^2)>a^2.
\]
This contradicts the middle-to-last constraint $\|x_2-x_1\|\le a$.
\end{proof}

\bibliographystyle{plainnat}
\bibliography{references}

\end{document}